\documentclass[12pt]{amsart}

\usepackage[frenchb, greek,english]{babel}
\usepackage[latin1]{inputenc}
\usepackage{amsmath,latexsym,amsmath}
\usepackage{amsfonts}
\usepackage{amssymb}
\usepackage{geometry}
\usepackage{graphicx}
\usepackage{amsthm}
\usepackage{bm}
\usepackage{setspace}
\usepackage[final]{showkeys}
\usepackage{dsfont}

\usepackage{epsfig}
\usepackage{graphicx}
\usepackage{psfrag}
\usepackage{enumitem}
\usepackage{bbm}
\usepackage{hyperref}

\theoremstyle{theorem}
\newtheorem{theorem}{Theorem}[section]
\newtheorem{definition}[theorem]{Definition}
\newtheorem{corollary}[theorem]{Corollary}
\newtheorem{proposition}[theorem]{Proposition}
\newtheorem{lemma}[theorem]{Lemma}
\theoremstyle{definition}
\newtheorem{example}[theorem]{Example}
\newtheorem{remark}[theorem]{Remark}
\newcommand{\R}{\mathbb{R}}
\newcommand{\lip}{{\rm Lip}}

\begin{document}

\title[Annealed and quenched limit theorems for expanding RDS]{Annealed and quenched limit theorems for random expanding dynamical systems}

\author{Romain Aimino}\address{
Aix Marseille Universit\'e, CNRS, CPT, UMR 7332, 13288 Marseille, France
Universit\'e de Toulon, CNRS, CPT, UMR 7332, 83957 La Garde, France
}
\email{\href{mailto:aimino@cpt.univ-mrs.fr}{aimino@cpt.univ-mrs.fr}}
\urladdr{\url{http://www.cpt.univ-mrs.fr/~aimino/}}

\author{Matthew Nicol}\address{
Department of Mathematics,
University of Houston,
Houston Texas,
USA}
\email{\href{mailto:nicol@math.uh.edu}{nicol@math.uh.edu}}
\urladdr{\url{http://www.math.uh.edu/~nicol/}}

\author{Sandro Vaienti}\address{
Aix Marseille Universit\'e, CNRS, CPT, UMR 7332, 13288 Marseille, France
Universit\'e de Toulon, CNRS, CPT, UMR 7332, 83957 La Garde, France
}
\email{\href{mailto:vaienti@cpt.univ-mrs.fr}{vaienti@cpt.univ-mrs.fr}}
\urladdr{\url{http://www.cpt.univ-mrs.fr/~vaienti/}}

\date{\today}

\thanks{RA was partially supported by Conseil R\'egional  Provence-Alpes-C\^ote d'Azur. RA and SV were supported by the ANR-
Project {\em Perturbations} and by
the PICS (Projet International de Coop\'eration Scientifique), Propri\'et\'es statistiques des syst\`emes dynamiques deterministes et al\'eatoires, with the University of Houston, n. PICS05968. SV thanks the University of Houston for supporting his visits during the completion of this work. MN was supported by the French CNRS with a {\em poste d'accueil} position  at the Center of Theoretical Physics in Luminy.
 MN was partially supported by NSF grant DMS 1101315. RA wishes to thank Carlangelo Liverani for discussions and encouragements.}

\maketitle
\begin{abstract}
In this paper, we investigate annealed and quenched limit theorems for random expanding dynamical systems. Making use of functional analytic techniques and more probabilistic arguments with martingales, we prove annealed versions of a central limit theorem, a large deviation principle, a local limit theorem, and an almost sure invariance principle. We also discuss the quenched central limit theorem, dynamical Borel-Cantelli lemmas, Erd\"{o}s-R\'enyi laws and concentration inequalities.
\end{abstract}


~
\\ 

{\small Appeared online on {\em Probability Theory and Related Fields}. The final publication is available at Springer via {\url{http://dx.doi.org/10.1007/s00440-014-0571-y}}.}

\section{Introduction}
\subsection{Limit theorems for Random Dynamical Systems: a brief survey}

Statistical properties for deterministic uniformly expanding dynamical systems are by now pretty well understood, starting from the existence of an absolutely continuous invariant probability \cite{LY, GB, S}, to exponential decay of correlations and limit theorems  \cite{HK, R-E, GH}, and more refined properties, such as Erd\"{o}s-R\'enyi laws \cite{CC,DN}, dynamical Borel-Cantelli lemmas \cite{CK, Kim} and concentration inequalities \cite{CMS, CG}. Most of these results are derived from the existence of a spectral gap of the transfer operator of the system, when acting on a appropriately chosen Banach space. The books \cite{BG_book} and \cite{Bal00} contain a nice overview and historical perspectives on the subject.

For random dynamical systems, the understanding of the situation is still unsatisfactory. A random dynamical system can be seen as a random composition of maps acting on the same space $X$, where maps are chosen according to a stationary process. When the process is a sequence of independent random maps, this gives rise to a Markov chain with state space $X$. The iid setting has been extensively studied in the book of Kifer \cite{K86}, while the general case is treated in the book of Arnold \cite{Ar}. The relevance of random dynamical systems is obvious from  the  fact that in most physical applications, it is very unlikely that the same map is iterated all along the time, and it is rather the case that different maps, very close to a fixed one, are iterated randomly. This topic of stochastic perturbations is very well covered in \cite{K88}. Random dynamical systems also arise naturally in the field of particle systems on lattices, where on a single site, the particle is subject to a local deterministic dynamic, but can jump from one site to another randomly, see \cite{KY, Tum}.

The existence of stationary measures absolutely continuous with respect to Lebesgue measure was first studied by Pelikan \cite{Pel} and Morita \cite{Mor1} in the case where $X$ is the unit interval, and in the thesis of Hsieh \cite{Hsieh} for the multidimensional case. This question has also been investigated for one-dimensional and multidimensional systems in the case of position dependent probabilities, see \cite{BaGo} and references therein. For limit theorems, the available literature is much more sparse. It should be first stressed that for random systems, limit theorems are of two kind : annealed results concern properties related to the skew-product dynamics, while quenched results describe properties for almost every realization. Annealed results follow from the spectral analysis of an annealed transfer operator, generalizing the successful approach for deterministic systems. In this line of spirit, we can cite the papers of Baladi \cite{Bal97}, Baladi-Young \cite{BY} or Ishitani \cite{Ishi} and the thesis \cite{Tum}. Quenched results are usually more difficult to prove. Exponential decay of correlation in a quenched regime has been proved using Birkhoff cones technique in \cite{BKS, Buz99, K08}, while a quenched central limit theorem and a law of iterated logarithm are studied by Kifer \cite{K98}, using a martingale approximation. These results deal with more general stationary process, where absolutely continuous stationary measure can fail to exist and are replaced by a family of sample measures. Closer to our setting are the papers \cite{AS, ALS} which are concerned with random toral automorphisms, and a very recent work of M. Stenlund and H. Sulky (A coupling approach to random circle maps expanding on the average, preprint, 2013), where quenched exponential decay of correlations together with an annealed almost sure invariance principle are shown for iid expanding circle maps, using the coupling method.

\subsection{Limit theorems: our new results}

When the random dynamical system is contracting on average, the transition operator of the Markov chain admits a spectral gap on a space of Hölder functions, from which one can deduce a large span of limit theorems following the Naga\"{e}v's method, see for instance \cite{HH2} and references therein. 
Nevertheless, for the applications we have in mind, the maps will instead be expanding on average. In this situation, the transition operator generally fails to admit a spectral gap and we will preferably rely on the quasi-compactness of an associated annealed transfer operator on an appropriate Banach space.
In this paper, we provide an abstract functional framework, valid for several one dimensional and multidimensional systems, under which annealed limit theorems hold for smooth enough observables. More precisely, under a spectral gap assumption for the annealed transfer operator, we apply Naga\"{e}v's perturbative method to obtain a central limit theorem with rate of convergence and a large deviation principle. A Borel-Cantelli argument allows us to derive immediately a quenched upper bound for the large deviation principle, but the question of whether a quenched lower bound holds remains open. We also show a local limit theorem under an abstract aperiodicity condition, and relate in most practical cases this condition to the usual one for individual maps. We apply Gou\"{e}zel's spectral method to prove an annealed almost sure invariance principle for vector valued observables : this is a strong reinforcement of the central limit theorem which has many  consequences, such as the law of the iterated logarithm, the  functional central limit theorem, and the  almost sure central limit theorem \cite{LP}.

Changing slightly our approach, we then adapt the martingale approximation method, which goes back to Gordin \cite{Go}, and give an alternative proof of the annealed central limit theorem. This requires the introduction of a symbolic deterministic system on which the standard martingale procedure can be pursued. Decay of annealed correlations is the key ingredient here and allows us to show that the Birkhoff's sums can be written as the sum of a backwards martingale and a coboundary, from which the central limit theorem follows from the analogous result for martingales.

We next investigate dynamical Borel-Cantelli lemmas : if $(f_n)$ is a bounded sequences of positive functions lying in the functional space, such as $\sum_n \int f_n d \mu = \infty$, where $\mu$ is the stationary measure, we prove that $$ \frac{\sum_{k=0}^{n-1} f_k(T_{\underline{\omega}}^k x)}{\sum_{k=0}^{n-1} \int f_k d \mu} \to 1,$$ for almost every realization $\underline{\omega}$ and almost every point $x \in X$, a property usually called strong Borel-Cantelli lemma in the literature. Of particular interest is the case where $f_n$ are the characteristic functions of a sequence of decreasing sets, since this relates to recurrence properties of the system. The proof builds upon annealed decay of correlations, and is a consequence of the work of Kim \cite{Kim}. This result can be seen as a generalization of the strong law of large numbers, and it is hence natural to study the nature of the fluctuations in this convergence. Provided we have precise enough estimates on the measure of the sets, we prove a central limit theorem. For this purpose, we employ the martingale technique already used before for Birkhoff sums, and make use of a central limit theorem for non stationary martingales from Hall and Heyde \cite{Hall_Heyde}, mimicking the proof from \cite{HNVZ} for the deterministic case.

We then turn to Erd\"{o}s-R\'enyi laws : these limit laws give information on the maximal average gain in the case where the length of the time window ensures there is a non-degenerate limit. This result was first formulated by Erd\"{o}s and R\'enyi \cite{ER}, and brought in dynamical context in \cite{CC, DN,Kessebohmer} among others. Making use of the large deviation principle, we adapt the proof of \cite{DN} to show that an annealed Erd\"{o}s-R\'enyi law holds true in the random situation, for one-dimensional transformations.

Importing a technique from the field of random walks in random environments, Ayyer, Liverani and Stenlund \cite{ALS} proved a quenched central limit theorem for random toral hyperbolic automorphisms. Their approach consists in proving a spectral gap for the original system and for a "doubled" system acting on $X^2$, where maps are given by $\hat{T}_{\omega} (x,y) = (T_{\omega} x, T_{\omega} y)$, and driven by the same iid process. This allows to prove a quenched central limit theorem for subsequences of the Birkhoff sums by a Borel-Cantelli argument, and the large deviation principle helps to estimate the error occurring in the gaps. Unfortunately, this method needs a precise relation between the asymptotic variance of the observable on the original system, and the asymptotic variance of a deduced observable on the doubled system. This relation is easily shown when all maps preserve the same measure, as it is the case in \cite{ALS}, but is harder to prove, and possibly false, in full generality. Hence, in this paper, we restrict our attention to the case where all maps preserve the Lebesgue measure on the unit interval. Apart the trivial case where all maps are piecewise onto and linear, we show that we can include in the random compositions a class of maps introduced in \cite{CHMV}, which have a neutral fixed point and a point where the derivative blows up. The general case remains open.

Concentration inequalities are a well known subject in probability theory, and have numerous and deep consequences in statistics. They concern deviations from the mean for non additive functionals and hence generalize large deviations estimates for ergodic sums. Furthermore, these inequalities are non-asymptotic. The price to pay is that they do not give precise asymptotics for the deviation function, in contrast to  the large deviations principle. They were introduced in dynamical systems by Collet, Martinez and Schmitt \cite{CMS} who prove an exponential inequality for uniformly piecewise expanding maps of the interval. The paper \cite{CG} covers a wide range of uniformly and non-uniformly expanding/hyperbolic dynamical systems which are modeled by a Young tower. For random dynamical systems, concentration inequalities were not previously studied. As far as the authors know, the only result available is \cite{Mal}, which covers the case of the observational noise. We attempt to fill this gap and prove an annealed exponential concentration inequality for randomly expanding systems on the interval, generalizing the approach of \cite{CMS}. We then give an application to the rate of convergence of the empirical measure to the stationary measure.

\subsection{Plan of the paper}
The paper is outlined as follows. In section \ref{framework}, we described our abstract functional framework, and give several classes of one dimensional and multidimensional examples which fit the assumptions. In section \ref{spectral} we apply Naga\"{e}v method to prove annealed limit theorems. In section \ref{clt_martingale}, we explain how the central limit theorem follows from a martingale approximation. In section \ref{borel_cantelli}, we prove dynamical Borel-Cantelli lemmas and a central limit theorem for the shrinking target problem. In section \ref{erdos_renyi}, we prove an Erd\"{o}s-R\'enyi law for random one-dimensional systems. In section \ref{quenched_clt}, we consider the quenched central limit theorem for specific one dimensional random systems. Finally, in section \ref{concentration}, we prove an exponential concentration inequality and discuss its applications.

~

The letter $C$ denotes a positive constant whose precise value has no particular importance and can change from one line to another.

\section{Abstract framework and examples} \label{framework}

Let $(\tilde{\Omega}, \tilde{\mathcal{T}}, \tilde{\mathbb{P}})$ be a probability space, and $\theta : \tilde{\Omega} \to \tilde{\Omega}$ be a measure preserving transformation. Let now $(X, \mathcal{A})$ be a measurable space. Suppose that to each $\underline{\omega} \in \tilde{\Omega}$ is associated a transformation $T_{\underline{\omega}} : X \to X$ such that the map $(\underline{\omega}, x) \mapsto T_{\underline{\omega}}(x)$ is measurable. We are then considering random orbits $T_{\theta^n \underline{\omega}} \circ \ldots \circ T_{\underline{\omega}} x$.

One can now define a skew-product transformation $F : \tilde{\Omega} \times X \to \tilde{\Omega} \times X$ by $F(\underline{\omega}, x) = (\theta \underline{\omega}, T_{\underline{\omega}} x)$. We will say that a probability measure $\mu$ on $(X,\mathcal{A})$ is a stationary measure if $\tilde{\mathbb{P}} \otimes \mu$ is invariant under $F$.

The simplest situation possible is the i.i.d. case : $(\tilde{\Omega}, \tilde{\mathcal{T}}, \tilde{\mathbb{P}})$ is a countable product space, namely $\tilde{\Omega} = \Omega^{\mathbb{N}}$, $\tilde{\mathcal{T}} = \mathcal{T}^{\otimes \mathbb{N}}$, $\tilde{\mathbb{P}} = \mathbb{P}^{\otimes \mathbb{N}}$ and $\theta$ is the full shift. If to each $\omega \in \Omega$ is associated a map $T_{\omega}$ on $X$ such that $(\omega,x) \mapsto T_{\omega}(x)$ is measurable, then we define $T_{\underline{\omega}} = T_{\omega_1}$ for each $\underline{\omega} \in \tilde{\Omega}$, with $\underline{\omega} = (\omega_1, \omega_2, \ldots)$. This fits the framework described previously. It is easily seen that $\mu$ is a stationary measure iff $\mu(A) = \int_{\Omega} \mu(T_{\omega}^{-1}(A)) \, d\mathbb{P}(\omega)$ for each $A \in \mathcal{A}$.
Moreover, if we set $X_n(\underline{\omega}, x)= T_{\theta^n \underline{\omega}} \circ \ldots \circ T_{\underline{\omega}} x,$ this defines a homogeneous Markov chain with state space  $(X, \mathcal{A})$ and transition operator given by $U(x, A)= \mathbb{P}(\{ \omega : T_{\omega}x \in A\})$ for any $x \in X$ and any set $A \in \mathcal{A}$. 

From now, we will always consider this i.i.d. situation. Suppose now that $(X,\mathcal{A})$ is endowed with a probability measure $m$ such that each transformation $T_{\omega}$ is non-singular w.r.t. $m$. We will investigate existence and statistical properties of stationary measures absolutely continuous w.r.t. $m$. To this end, we introduce averaged transfer and Koopman operators.

Since every transformation $T_\omega$ is non-singular, the transfer operator $P_\omega$ and the  Koopman operator $U_\omega$ of $T_\omega$ are well defined, and act respectively on $L^1(m)$ and $L^{\infty}(m)$. We recall their definitions for the convenience of the reader, and refer to \cite{BG_book} and \cite{Bal00} for more properties. The Koopman operator of $T_{\omega}$ is acting on $L^{\infty}(m)$ by $U_{\omega} f = f \circ T_{\omega}$. Its action on conjugacy classes of functions is well defined since $T_{\omega}$ is non-singular w.r.t. $m$. The transfer operator, or Perron-Frobenius operator is acting on $L^1(m)$ in the following way : for $f \in L^1(m)$, define the complex measure $m_f$ by $d m_f = f dm$. Then $P_{\omega} f$ is defined to be the Radon-Nykodim derivative of the push-forward measure $T_{\star} m_f$ w.r.t. $m$, which is well defined by non-singularity. The main relation between these two operators is given by the duality formula $\int_X P_{\omega}f (x) g(x) dm(x) = \int_X f(x) U_{\omega}g(x) dm(x)$, which holds for all $f\in L^1(m)$ and $g \in L^{\infty}(m)$.

We can now defined the averaged versions of these operators. For $f \in L^1(m)$, we define $Pf$ by the formula $Pf(x) = \int_{\Omega} P_\omega f(x) \,d \mathbb{P}(\omega)$, and for $g \in L^{\infty}(m)$, we define $Ug$ by $Ug(x) = \int_\Omega U_\omega g(x) \, d\mathbb{P}(\omega)$. The operator $U$ just defined coincides with the transition operator $U$ of the Markov chain $(X_n),$ when acting on functions. Notice that for all $n \ge 0$ and $g \in L^{\infty}(m)$, one has $U^n g(x) = \int_{\Omega^n} g(T_{\omega_n} \ldots T_{\omega_1} x) \, d\mathbb{P}^{\otimes n} (\omega_1, \ldots, \omega_n) = \int_{\tilde{\Omega}} g(T_{\omega_n} \ldots T_{\omega_1} x) \, d \tilde{\mathbb{P}}(\underline{\omega})$, because $\tilde{\mathbb{P}}$ is a product measure.  It is then straightforward to check that $U$ is the dual operator of $P$, that is $\int_X Pf(x) g(x) \, dm(x) = \int_X f(x) Ug(x) \, dm(x)$ for all $f \in L^1(m)$ and $g \in L^{\infty}(m)$. An absolutely continuous probability measure is stationary iff its density is a fixed point of $P$.

We will assume that $P$ has good spectral properties on some Banach space of functions. More precisely, we assume that there exists a Banach space $(\mathcal{B}, \| . \|)$ such that :

\begin{enumerate}

\item \label{compact_embed} $\mathcal{B}$ is compactly embedded in $L^1(m)$;

\item \label{1_in_B} Constant functions lie in $\mathcal{B}$;


\item \label{banach_lattice} $\mathcal{B}$ is a complex Banach lattice : for all $f \in \mathcal{B}$, $\left | f \right|$ and $\bar{f}$ belong to $\mathcal{B}$;

\item \label{Stab} $\mathcal{B}$ is stable under $P$ : $P(\mathcal{B}) \subset \mathcal{B}$, and $P$ acts continuously on $\mathcal{B}$;

\item \label{LY} $P$ satisfies a Lasota-Yorke inequality : there exist $N \ge 1$, $\rho < 1$ and $K \ge 0$ such that $\| P^N f \| \le \rho \| f \| +K \|f\|_{L^1_m}$ for all $f \in \mathcal{B}$.

\end{enumerate}

The LY inequality implies in particular that the spectral radius of $P$ acting on $\mathcal{B}$ is less or equal than $1$, and since $m$ belongs to the topological dual of $\mathcal{B}$ by the first assertion, and is fixed by $P^{\star}$ the adjoint of $P$, we have that the spectral radius is $1$. Hence, by Ionescu-Tulcea and Marinescu's theorem \cite{ITM} (see also \cite{H}), we have that the essential spectral radius of $P$ is less or equal than $\rho < 1$, implying that $P$ is quasi-compact on $\mathcal{B}$, since it has spectral radius $1$. A standard argument using compactness proves that $P^N$ and hence $P$ has a positive fixed point : there is an element $h \in \mathcal{B}$ with $Ph = h$, $h \ge 0$ and $\int_X h \, dm = 1$. As a consequence $1$ is an eigenvalue of $P$. Another consequence of quasi-compactness is the fact that the spectrum of $P$ is constituted of a finite set of eigenvalues with modulus 1 of finite multiplicity and the remaining spectrum is contained in a disk of radius strictly less than $1$. We will make the following assumption, that prevents the possibility of peripheral spectrum :

\begin{enumerate}
\item[6.] \label{spectral_gap} $1$ is a simple (isolated) eigenvalue of $P$, and there is no other eigenvalue on the unit circle.
\end{enumerate}

This assumption implies in particular that the absolutely continuous stationary measure is unique. We will denote it by $\mu$, and its density by $h$, throughout the paper.

Usually, assertions \ref{Stab} and \ref{LY} can be deduced from corresponding assertions for the operators $P_\omega$ if the constants appearing in the Lasota-Yorke inequality are uniform. Nevertheless, they can be established even if one of the maps $T_{\omega}$ is not uniformly expanding, as showed by the following class of examples :

\begin{example}[Piecewise expanding one-dimensional maps] ~  \\
A Lasota-Yorke map is a piecewise $C^2$ map $T : [0,1] \to [0,1]$ for which $\lambda(T) := \inf | T' | > 0$.

We denote by $P_T$ the transfer operator (with respect to Lebesgue measure) associated to $T$. One has $$P_Tf(x) = \sum_{Ty = x} \frac{f(y)}{|T'(y)|}$$ for all $f \in L^1(m)$. We will analyze the spectral properties of $P_T$ acting on the space of functions of bounded variation. We recall the definition. A function $f : [0,1] \to \mathbb{C}$ is of bounded variation if its total variation defined as $${\rm Var}(f) = \sup \sum_{i=0}^{n-1} |f(x_{i+1}) - f(x_i)|,$$ where the supremum is taken over all the finite partitions $0 = x_0 < \ldots < x_n = 1$, is finite.

For a class of equivalence $f \in L^1(m)$, we then define $${\rm Var}(f) = \inf \{ {\rm Var}(g) \, / \, f =g ~ m {\rm -ae} \}.$$

The space ${\rm BV} = \{f \in L^1(m)\, / \, {\rm Var}(f) < \infty\}$ is endowed with the norm $\|f \| = \|f\|_{L^1_m} + {\rm Var}(f)$, which turns it into a Banach space satisfying assumptions \ref{compact_embed}, \ref{1_in_B} and \ref{banach_lattice} above. Furthermore, this is a Banach algebra which embeds continuously into $L^{\infty}(m)$.

If $T$ is a Lasota-Yorke map, the following inequality holds :

\begin{proposition}[Lasota-Yorke inequality \cite{LY}]  For any $f \in {\rm BV}$, we have $${\rm Var}(P_T f) \le \frac{2}{\lambda(T)} {\rm Var}(f) + A(T) \|f\|_{L^1_m}$$ where $A(T)$ is a finite constant depending only on $T$.

\end{proposition}

Let $\Omega$ be a finite set, together with a probability vector $\mathbb{P} = \{p_{\omega}\}_{\omega \in \Omega}$ and a finite number of Lasota-Yorke maps $T = \{T_{\omega}\}_{\omega \in \Omega}$. We assume that $p_{\omega}>0$ for all $\omega \in \Omega$. The system $(\Omega, \mathbb{P}, T)$ is called a random Lasota-Yorke system.

The random Lasota-Yorke system $(\Omega, \mathbb{P}, T)$ is expanding in mean if $$\Lambda := \sum_{\omega \in \Omega} \frac{p_{\omega}}{\lambda(T_{\omega})} < 1.$$

The annealed transfer operator associated to $(\Omega, \mathbb{P}, T)$ is $P = \sum_{\omega \in \Omega} p_{\omega} P_{T_{\omega}}$. It satisfies $P^n = \sum_{\underline{\omega} \in \Omega^n} p_{\underline{\omega}}^n P_{T_{\underline{\omega}}^n}$, where $\underline{\omega} = (\omega_1, \ldots, \omega_n)$, $p_{\underline{\omega}}^n = p_{\omega_1} \ldots p_{\omega_n}$ and $T_{\underline{\omega}}^n = T_{\omega_n} \circ \ldots \circ T_{\omega_1}$.

\begin{proposition} \label{gapLY}
If $(\Omega, \mathbb{P}, T)$ is expanding in mean, then some iterate of the annealed transfer operator satisfies a Lasota-Yorke inequality on ${\rm BV}$.
\end{proposition}

\begin{proof}
By the classical LY inequality and subadditivity of the total variation, one has $${\rm Var}(P^nf) \le 2 \theta_n {\rm Var}(f) + A_n \|f \|_{L^1_m}$$ for all $n \ge 1$ and all $f \in {\rm BV}$, where $\theta_n = \sum_{\underline{\omega} \in \Omega^n} \frac{p_{\underline{\omega}}^n}{\lambda(T_{\underline{\omega}}^n)}$ and $A_n = \sum_{\underline{\omega} \in \Omega^n} p_{\underline{\omega}}^n A(T_{\underline{\omega}}^n)$. Since $\lambda(T_{\underline{\omega}}^n) \ge \lambda(T_{\omega_1}) \ldots \lambda(T_{\omega_n})$, one obtains $\theta_n \le \sum_{\underline{\omega} \in \Omega^n} \frac{p_{\omega_1}  \ldots p_{\omega_n}}{\lambda(T_{\omega_1}) \ldots \lambda(T_{\omega_n})} = \Lambda^n$. Hence, $2 \theta_n <1$ for $n$ large enough, while the corresponding $A_n$ is finite. This concludes the proof. \qed
\end{proof}

This implies assumptions \ref{Stab} and \ref{LY}.

\begin{remark}

\begin{enumerate}

\item Pelikan \cite{Pel} showed that the previous Lasota-Yorke inequality still holds under the weaker assumption $\sup_{x} \sum_{\omega \in \Omega} \frac{p_{\omega}}{|T_{\omega}'(x)|} < 1$.

\item The result is still valid if the set $\Omega$ is infinite, assuming an integrability condition for the distortion. See Remark 5.1 in Morita \cite{Mor1}.

\end{enumerate}

\end{remark}

From Ionescu-Tulcea and Marinescu theorem, it follows that the annealed transfer operator has the following spectral decomposition : $$P = \sum_{i} \lambda_i \Pi_i + Q,$$ where all $\lambda_i$ are eigenvalues of $P$ of modulus $1$, $\Pi_i$ are finite-rank projectors onto the associated eigenspaces, $Q$ is a bounded operator with a spectral radius strictly less than $1$. They satisfy $$\Pi_i \Pi_j = \delta_{ij} \Pi_i, \: Q \Pi_i = \Pi_i Q = 0.$$

This implies existence of an absolutely continuous stationary measure, with density belonging to ${\rm BV}$. Standard techniques show that $1$ is an eigenvalue and that the peripheral spectrum is completely cyclic. We'll give a concrete criterion ensuring that $1$ is a simple eigenvalue of $P$, and that there is no other peripheral eigenvalue, hence implying assumption 6. In this case, we will say that $(\Omega, \mathbb{P}, T)$ is mixing.

\begin{definition}
The random LY system $(\Omega, \mathbb{P}, T)$ is said to have the Random Covering (RC) property if for any non-trivial subinterval $I \subset [0,1]$, there exist $n \ge 1$ and $\underline{\omega} \in \Omega^n$ such that $T_{\underline{\omega}}^n(I) = [0,1]$.
\end{definition}

\begin{proposition} \label{rc_1D}   If $(\Omega, \mathbb{P}, T)$ is expanding in mean and has the (RC) property, then $(\Omega, \mathbb{P}, T)$ is mixing and the density of the unique a.c. stationary measure is bounded away from $0$.
\end{proposition}

\begin{proof} Since the peripheral spectrum of $P$ consists of a finite union of finite cyclic groups, there exists $k \ge 1$ such that $1$ is the unique peripheral eigenvalue of $P^k$. It suffices then to show that the corresponding eigenspace is one-dimensional. Standard arguments show there exists a basis of positive eigenvectors for this subspace, with disjoint supports. Let then $h\in {\rm BV}$ non-zero satisfying $h \ge 0$ and $P^kh = h$. There exist a non-trivial interval $I \subset [0,1]$ and $\alpha > 0$ such that $h \ge \alpha \mathds{1}_I$. Choose $n \ge 1$ and $\underline{\omega}^{\star} \in \Omega^{nk}$ such that $T_{\underline{\omega}^{\star}}^{nk}(I) = [0,1]$. For all $x \in [0,1]$, we have $$h(x) = P^{nk} h(x) \ge \alpha P^{nk} \mathds{1}_I(x) = \alpha \sum_{\underline{\omega} \in \Omega^{nk}} p_{\underline{\omega}}^{nk} \sum_{T_{\underline{\omega}}^{nk} y =x} \frac{\mathds{1}_I(y)}{| (T_{\underline{\omega}}^{nk})'(y)|} \ge \alpha p_{\underline{\omega}^{\star}}^{nk} \sum_{T_{\underline{\omega}^{\star}}^{nk} y =x} \frac{\mathds{1}_I(y)}{| (T_{\underline{\omega}^{\star}}^{nk})'(y)|}.$$ This shows that $h(x) \ge \alpha \frac{p_{\underline{\omega}^{\star}}^{nk}}{\|(T_{\underline{\omega}^{\star}}^{nk})' \|_{\rm sup}} > 0$, since there is always a $y \in I$ with $T_{\underline{\omega}^{\star}}^{nk} y =x$. This implies that $h$ has full support, and concludes the proof. \qed
\end{proof}

Some statistical properties of random one-dimensional systems, using the space ${\rm BV}$, were studied in the thesis of T\"{u}mel \cite{Tum}.
\end{example}

\begin{example}[Piecewise expanding multidimensional maps] ~ \\
We describe a class of piecewise expanding multi-dimensional maps introduced by Saussol \cite{S}. Denote by $m_d$ the $d$-dimensional Lebesgue measure, by $d(.,.)$ the euclidean distance and by $\gamma_d$ the $m_d$-volume of the unit ball of $\mathbb{R}^d$. Let $M$ a compact regular subset of $\mathbb{R}^d$ and let $T : M \to M$ be a map such that there exists a finite family of disjoint open sets $U_i \subset M$ and $V_i$ with $\bar{U_i} \subset V_i$ and maps $T_i : V_i \to \mathbb{R}^d$ satisfying for some $0 < \alpha \le 1$ and some small enough $\epsilon_0 > 0$ :

\begin{enumerate}

\item $m_d(M \setminus \cup_i U_i) = 0$;

\item for all $i$, the restriction to $U_i$ of $T$ and $T_i$ coincide, and $B_{\epsilon_0}(TU_i) \subset T_i(V_i)$;

\item for all $i$, $T_i$ is a $C^1$-diffeomorphism from $V_i$ onto $T_i V_i$, and for all $x,y \in V_i$ with $d(T_i x,T_i y) \le \epsilon_0$, we have $$\left| {\rm det} DT_i(x) - {\rm det} DT_i(y) \right| \le c \left| {\rm det} DT_i(x)\right| d(T_i x, T_iy)^{\alpha},$$ for some constant $c> 0$ independant of $i$, $x$ and $y$;

\item there exists $s < 1$ such that for all $x,y \in V_i$ with $d(T_i x, T_i y) \le \epsilon_0$, we have $d(x,y) \le s d(T_i x, T_i y)$;

\item Assume that the boundaries of the $U_i$ are included in piecewise $C^1$ codimension one embedded compact submanifolds. Define  $$Y = \sup_x \sum_i \sharp \{ {\rm smooth ~ pieces ~ intersecting ~} \partial U_i {\rm ~ containing ~} x\},$$
  and $$\eta_0 = s^{\alpha} + \frac{4s}{1-s} Y \frac{\gamma_{d-1}}{\gamma_d}.$$ Then $\eta_0 < 1$.

\end{enumerate}

The above conditions can be weakened in order to allow infinitely many domains of injectivity and also the possibility of fractal boundaries. We refer the interested reader to \cite{S} for more details.

We will call such maps piecewise expanding maps in the sense of Saussol. The analysis of the transfer operators of these maps requires the introduction of the functional space called Quasi-H\"{o}lder. This space was first defined and studied by Keller \cite{Kel} for one-dimensional transformations, and then extended to multidimensional systems by Saussol \cite{S}.

We give the definition of this space. Let $f : \mathbb{R}^d \to \mathbb{C}$ be a measurable function. For a Borel subset $A \subset  \mathbb{R}^d$, define $\mbox{osc}(f, A) = \underset{x,  y \in A}{\rm ess ~ sup} |f(x) -f(y)|$. For any $\epsilon > 0$, the map $x \mapsto \mbox{osc}(f, B_{\epsilon}(x))$ is a positive lower semi-continuous function, so that the integral $\int_{\mathbb{R}^d} \mbox{osc}(f, B_{\epsilon}(x)) dx$ makes sense.
For $f \in L^1(\mathbb{R}^d)$ and $0 < \alpha \le 1$, define $$|f|_{\alpha} = \sup_{0 < \epsilon \le \epsilon_0} \frac{1}{\epsilon^{\alpha}} \int_{\mathbb{R}^d} \mbox{osc}(f, B_{\epsilon}(x)) dx.$$ For a regular compact subset $M \subset \mathbb{R}^d$, define $$V_{\alpha}(M) = \{f  \in L^1(\mathbb{R}^d) \, / \, {\rm supp}(f) \subset M, \, |f|_{\alpha} < \infty \},$$ endowed with the norm $\|f\|_{\alpha} = \|f \|_{L^1_m} + |f|_{\alpha}$, where $m$ is the Lebesgue measure normalized so that $m(M) = 1$. Note that while the norm depends on $\epsilon_0$, the space $V_{\alpha}$ does not, and two choices of $\epsilon_0$ give rise to two equivalent norms.

If $T$ is a piecewise expanding map in the sense of Saussol and $P_T$ is the transfer operator of $T$, a Lasota-Yorke type inequality holds :

\begin{proposition}(\cite[Lemma 4.1]{S})
Provided $\epsilon_0$ is small enough, there exists $\eta < 1$ and $D < \infty$ such that for any $f \in V_{\alpha}$, $$|P_T f|_{\alpha} \le \eta |f|_{\alpha} + D \|f\|_{L^1_m}.$$

\end{proposition}

Suppose now that $\Omega$ is a finite set, $\mathbb{P} = \{p_{\omega}\}_{\omega \in \Omega}$ a probability vector, and $\{T_{\omega}\}_{\omega \in \Omega }$ a finite collection of piecewise expanding maps on $M \subset \mathbb{R}^d$. This will be referred to as a random piecewise expanding multidimensional system. Take $\epsilon_0$ small enough so that the inequalities $|P_{T_{\omega}}|_{\alpha} \le \eta_{\omega} |f|_{\alpha} + D_{\omega} \|f\|_{L^1_m}$ for all $f \in V_{\alpha}$. Defining $\eta = \max \eta_{\omega}$ and $D = \max D_{\omega}$, so that $\eta < 1$ and $D < \infty$. Since $P = \sum_{\omega} P_{T_{\omega}}$, we immediately get $|P f|_{\alpha} \le \eta |f|_{\alpha} + \| f\|_{L^1_m}$, for all $f \in V_{\alpha}$. This shows that our abstract assumptions \ref{compact_embed} to \ref{LY} are all satisfied. To prove that $P$ is mixing, and hence check assumption 6, we can proceed as in the one-dimensional situation, introducing the same notion of random-covering.. Indeed, any positive non-zero element $h \in V_{\alpha}$ is bounded uniformly away from zero on some ball by Lemma 3.1 in \cite{S}, so we can mimic the proof of Proposition \ref{rc_1D} and get :

\begin{proposition} \label{rc_multi_d}
If $(\Omega, \mathbb{P}, T)$ is a random piecewise expanding multidimensional system which has the random covering property in the sense that for all ball $B \subset M$, there exists $n \ge 1$ and $\underline{\omega} \in \Omega^n$ such that $T_{\underline{\omega}}^n (B) = M$, then $(\Omega, \mathbb{P}, T)$ is mixing and the density of the unique a.c. stationary measure is bounded away from $0$.
\end{proposition}

There 	are  alternative functional spaces to study multidimensional expanding maps. One of them is the space of functions of bounded variations in higher dimension. Applications of this space to dynamical systems have been widely studied, see \cite{GB, Liv13} among many others. We also mention the thesis of Hsieh \cite{Hsieh} who investigates the application of this space to random multidimensional maps, using the same setting as us.  Notice that maps studied there are the so-called Jab\l o\'nski maps, for which the dynamical partition is made of rectangles. This kind of maps will appear later in this paper, when we will derive a quenched CLT. Nevertheless, the space BV in higher dimensions presents some drawbacks : it is not included in $L^{\infty}$ and there exist some positive functions which are not bounded below on any ball, making the application of random covering difficult, in contrast to the the Quasi-H\"{o}lder space. Apart multidimensional BV, another possibility is to use fractional Sobolev spaces, as done in a deterministic setting by Thomine \cite{Tho}.
\end{example}

\begin{example}[Random expanding piecewise linear maps] ~ \\
Building on a work by Tsujii \cite{Tsu}, we considerer random compositions of piecewise linear maps. First recall a definition :
\begin{definition}

Let $U$ be a bounded polyhedron in $\mathbb{R}^d$ with non-empty interior. An expanding piecewise linear map on $U$ is a combination $(\mathcal{T}, \mathcal{U})$ of a map $\mathcal{T} : U \to U$ and a family $\mathcal{U} = \{U_k\}_{k=1}^l$ of polyhedra $U_k \subset U$, $k = 1, \ldots, l$, satisfying the conditions

\begin{enumerate}

\item the interiors of polyhedra $U_k$ are mutually disjoint,

\item $\cup_{k=1}^l U_k = U$,

\item the restriction of the map $\mathcal{T}$ to the interior of each $U_k$ is an affine map and

\item there exists a constant $\rho > 1$ such that $\|D\mathcal{T}_x(v)\| \ge \rho \|v\|$ for all $x \in \cup_{k=1}^l {\rm int}(U_k)$, and all $v \in \mathbb{R}^d$.

\end{enumerate}

\end{definition}

We will drop $\mathcal{U}$, writing merely $\mathcal{T}$, when the partition $\mathcal{U}$ is understood. A basic consequence of Tsujii \cite{Tsu}, using the Quasi-H\"{o}lder space, is the following :

\begin{proposition} For any expanding piecewise linear map $\mathcal{T}$ on $U$, there exists constant $\epsilon_0> 0$, $\theta< 1$, $C, K > 0$ such that, for any $n\ge 0$ and $f \in V_1$ : $$|P_{\mathcal{T}}^n f |_1 \le C \theta^n |f|_1 + K\|f \|_{L^1_m},$$ where $P_{\mathcal{T}}$ is the transfer operator of $\mathcal{T}$.
\end{proposition}

Let $T = \{(\mathcal{T}_{\omega}, \mathcal{U}_{\omega})\}_{\omega \in \Omega}$ be a finite collection of expanding piecewise linear map on $U$, and $ \mathbb{P} = \{p_{\omega}\}_{\omega \in \Omega}$ a probability vector.

Choosing $\epsilon_0> 0$, $\theta < 1$, $C, D < \infty$ adequately, we get $|P_{\mathcal{T}}^n f |_1 \le C \theta^n |f|_1 + K\|f \|_{L^1_m}$ for all $f \in V_1$, where $P$ is the annealed transfer operator. If we assume furthermore that the system has the random covering property, then the Proposition \ref{rc_multi_d} also holds true.

\end{example}

\section{Spectral results} \label{spectral}

We assume here that there exists a Banach space $\mathcal{B}_0 \subset L^1(m)$, with norm $\| . \|_0$, and a constant $C >0$ such that $\|f g \| \le C \|f \|_0 \|g \|$ for all $f \in \mathcal{B}_0$ and $g \in \mathcal{B}$. This assumption is clearly satisfied with $\mathcal{B} = \mathcal{B}_0$ if $\mathcal{B}$ is a Banach algebra, as it is the case for the space of functions of bounded variation in one dimension, or the Quasi-H\"{o}lder space. If $\mathcal{B}$ is the space of functions of bounded variation in $\mathbb{R}^d$, then we can take $\mathcal{B}_0 = {\rm Lip}$, see lemma 6.4 in \cite{Tho}. Elements of $\mathcal{B}_0$ will play the role of observables in the following.

The spectral decomposition of $P$ yields $P = \Pi + Q$ where $\Pi$ is the projection given by $\Pi f  = \left( \int_X f \, dm \right) h$ and $Q$ has spectral radius on $\mathcal{B}$ strictly less than $1$ and satisfies $\Pi Q = Q \Pi = 0$. It follows that $P^n = \Pi + Q^n$, where $\|Q^n \| \le C \lambda^n$, for some $C \ge 0$ and $\lambda < 1$. This implies exponential decay of correlations :

\begin{proposition} \label{decay}  We have :

\begin{enumerate}

\item  For all $f \in \mathcal{B}_0$ and $g \in L^{\infty}(m)$, $$\left| \int_X f \, U^n g \, d\mu - \int_X f \, d\mu \int_X g \, d \mu \right| \le C \lambda^n \|f \|_0 \|g \|_{L^{\infty}_m}, $$

\item If $\mathcal{B}$ is continuously embedded in $L^{\infty}(m)$, then for all $f \in \mathcal{B}_0$ and $g \in L^1(m)$, $$\left| \int_X f \, U^n g \, d\mu - \int_X f \, d\mu \int_X g \, d \mu \right| \le C \lambda^n \|f \|_0 \|g \|_{L^1_m},$$

\item If $\mathcal{B}$ is continuously embedded in $L^{\infty}(m)$ and if the density of $\mu$ is bounded uniformly away from 0, then for all $f \in \mathcal{B}_0$ and $g \in L^1(\mu)$
 $$\left| \int_X f \, U^n g \, d\mu - \int_X f \, d\mu \int_X g \, d \mu \right| \le C \lambda^n \|f \|_0 \|g \|_{L^1_\mu}.$$
\end{enumerate}

\end{proposition}
The proof is classical, see appendix C.4 in \cite{AFLV} for the deterministic analogue.

We will now investigate limit theorems, namely a Central Limit Theorem (CLT) and a Large Deviation Principle (LDP), following Nagaev's perturbative approach. We refer to \cite{HH} for a full account of the theory. Let $\varphi \in \mathcal{B}_0$ be a bounded real observable with $\int_X \varphi \, d\mu = 0$. Define $X_k$ on $\tilde{\Omega} \times X$ by $X_k(\underline{\omega}, x) = \varphi(T_{\omega_k} \ldots T_{\omega_1} x)$ and $S_n = \sum_{k=0}^{n-1} X_k$. The first step is to prove the existence of the asymptotic variance.

\begin{proposition}  \label{green_kubo}

The limit $\sigma^2 = \lim_{n \to \infty} \frac{1}{n} \mathbb{E}_{\tilde{\mathbb{P}} \otimes \mu} (S_n^2)$ exists, and is equal to $$\sigma^2 = \int_X \varphi^2 \, d\mu + 2 \sum_{n=1}^{+\infty} \int_X \varphi \, U^n \varphi \, d \mu.$$

\end{proposition}

\begin{proof}
We expand the term $S_n^2$ and get $\mathbb{E}_{\tilde{\mathbb{P}} \otimes \mu} (S_n^2) = \sum_{k,l = 0}^{n-1} \mathbb{E}_{\tilde{\mathbb{P}} \otimes \mu} (X_k X_l)$.

\begin{lemma}
For all integers $k$ and $l$, one has $\mathbb{E}_{\tilde{\mathbb{P}} \otimes \mu}(X_k X_l) = \int_X \varphi \, U^{|k-l|} \varphi \, d\mu$.
\end{lemma}

\begin{proof}[Proof of the lemma]
By symmetry, we can assume $k \ge l$. We have $$\begin{aligned} &\mathbb{E}_{\tilde{\mathbb{P}} \otimes \mu}(X_k X_l) = \int_X h(x) \int _{\tilde{\Omega}} X_k(\underline{\omega}, x) X_l(\underline{\omega}, x) \, d\tilde{\mathbb{P}}(\underline{\omega}) dm(x)&  \\ &= \int_X h(x) \int_{\tilde{\Omega}} (\varphi \circ T_{\omega_k} \circ \ldots \circ T_{\omega_{l+1}}) (T_{\omega_l}  \ldots T_{\omega_1} x) \varphi(T_{\omega_l} \ldots T_{\omega_1} x) \, d\tilde{\mathbb{P}}(\underline{\omega}) dm(x)&  \\ & = \int_X h(x) \int_{\tilde{\Omega}} U^l(\varphi (\varphi \circ T_{\omega_k} \circ \ldots \circ T_{\omega_{l+1}}))(x) \, d\tilde{\mathbb{P}}(\omega_{l+1}, \ldots) dm(x)& \\ &= \int_{\tilde{\Omega}}  \int_X P^l h(x) \varphi(x) \varphi(T_{\omega_k} \ldots  T_{\omega_{l+1}}x) \, dm(x) d\tilde{\mathbb{P}}(\omega_{l+1}, \ldots) = \int_X \varphi(x) U^{k-l} \varphi(x) \, d \mu (x)& \end{aligned}$$ \qed

\end{proof}

\noindent Applying this lemma, we get $$\mathbb{E}_{\tilde{\mathbb{P}} \otimes \mu} (S_n^2) = \sum_{k,l=0}^{n-1} \int_X \varphi \, U^{|k-l|} \varphi \, d \mu = n \int_X \varphi^2 \, d\mu + 2 \sum_{k=1}^n (n-k) \int_X \varphi \, U^k \varphi \, d\mu. $$

\noindent Since $\int_X \varphi \, U^k \varphi \, d\mu$ decays exponentially fast, we see immediately that $\frac{1}{n} \mathbb{E}_{\tilde{\mathbb{P}} \otimes \mu}(S_n^2) $ goes to the desired quantity. \qed

\end{proof}

Let us mention that we have the following criteria to determine whether the asymptotic variance is $0$. The proof follows along the same lines as Lemma 4.1 in \cite{ALS}.

\begin{proposition}
The asymptotic variance satisfies $\sigma^2 = 0$ if and only if there exists $\psi \in L^2(\mu)$ such that, for $\mathbb{P}$-a.e. $\omega$, $\varphi = \psi - \psi \circ T_{\omega}$ $\mu$-a.e.

\end{proposition}

Denote by $\mathcal{M}_{\mathcal{B}}$ the set of all probability measures on $(X,\mathcal{A})$ which are absolutely continuous w.r.t. $m$, and whose density lies in $\mathcal{B}$. For a measure $\nu \in \mathcal{M}_{\mathcal{B}}$, we will denote by $\| \nu \|$ the $\mathcal{B}$-norm of the density $\frac{d \nu}{dm}$. Now, we are able to state the main theorems of this section :

\begin{theorem}[Central Limit Theorem] \label{CLT_spec}
For every probability measure $\nu \in \mathcal{M}_{\mathcal{B}}$, the process $(\frac{S_n}{\sqrt{n}})_n$ converges in law to $\mathcal{N}(0,\sigma^2)$ under the probability $\tilde{\mathbb{P}} \otimes \nu$.

\end{theorem}

\begin{theorem}[Large Deviation Principle] \label{LDP_spec}
Suppose that $\sigma^2 > 0$. Then there exists a non-negative rate function $c$, continuous, strictly convex, vanishing only at $0$, such that for every $\nu \in \mathcal{M}_{\mathcal{B}}$ and every sufficiently small $\epsilon >0$, we have $$\lim_{n \to \infty} \frac{1}{n} \log \tilde{\mathbb{P}} \otimes \nu (S_n > n \epsilon) = - c(\epsilon)$$

\end{theorem}

\noindent In particular, these theorems are valid for both the reference measure $m$ and the stationary one $\mu$, with the same asymptotic variance and the same rate function.

We introduce Laplace operators, which will encode the moment-generating function of the process. For every $z \in \mathbb{C}$, we define $P_z$ by $P_z(f) = P(e^{z \varphi} f)$. Thanks to our assumption on $\mathcal{B}_0$, this a well defined and continuous operator on $\mathcal{B}$, and the map $z \mapsto P_z$ is complex-analytic on $\mathbb{C}$ : indeed, if we define $C_n(f) = P(\varphi^n f)$, then  $P_z = \sum_{n \ge 0} \frac{z^n}{n!}C_n$, and this series is convergent on the whole complex plane since $\|C_n \| \le (C \| \varphi\|_0)^n \| P\|$.


We have the following fundamental relation :

\begin{lemma} \label{fund}
For every $n \ge 0$ and every $f \in \mathcal{B}$, we have $$\int_{\tilde{\Omega}} \int_X e^{z S_n( \underline{\omega}, x)} f(x) \, dm(x) \, d\tilde{\mathbb{P}}(\underline{\omega}) = \int_X P^n_z(f) \, dm.$$
\end{lemma}

\begin{proof}
We proceed by induction on $n$. The case $n = 0$ is trivial. Assume that the relation is valid for some $n \ge 0$, and all $f \in \mathcal{B}$. Let $f$ be a member of $\mathcal{B}$. Since $P_z(f)$ belongs to $\mathcal{B}$, the induction hypothesis gives $$ \begin{aligned}  \int_X P_z^{n+1}(f) \, dm &= \int_X P^n_z(P_z(f)) \, dm = \int_{\tilde{\Omega}} \int_X e^{z S_n(\underline{\omega}, x)} P_z(f)(x) \, dm(x) d \tilde{\mathbb{P}}( \underline{\omega})& \\ &= \int_{\tilde{\Omega}} \int_X e^{z S_n(\underline{\omega}, x)} P(e^{z \varphi}f)(x) \, dm(x) d \tilde{\mathbb{P}}( \underline{\omega})& \\ & = \int_{\tilde{\Omega}} \int_X U(e^{z S_n(\underline{\omega}, \, . \,)})(x) e^{z \varphi(x)} f(x) \, dm(x) d \tilde{\mathbb{P}}( \underline{\omega})& \\ &= \int_X \int_{\tilde{\Omega}} \int_{\Omega} e^{z (\varphi(x) + S_n(\underline{\omega}, T_{\omega} x))} \, d\mathbb{P}(\omega) d \tilde{\mathbb{P}}( \underline{\omega}) f(x)  dm(x)&
 \end{aligned}$$
But $\varphi(x) + S_n(\underline{\omega}, T_\omega x) = S_{n+1}(\omega \underline{\omega}, x)$, where $\omega \underline{\omega}$ stands for the concatenation $(\omega, \omega_1, \omega_2, \ldots)$ if $\underline{\omega} = (\omega_1, \omega_2, \ldots)$. As $\int_{\tilde{\Omega}} \int_{\Omega} e^{z S_{n+1}( \omega \underline{\omega}, x)}  \, d\mathbb{P}(\omega) d \tilde{\mathbb{P}}( \underline{\omega}) = \int_{\tilde{\Omega}} e^{z S_{n+1}(\underline{\omega}, x)} d \tilde{\mathbb{P}}( \underline{\omega})$ because of the product structure of $\tilde{\mathbb{P}}$, we obtain the formula for $n+1$ and $f$. \qed

\end{proof}

If $f$ is the density w.r.t. $m$ of a probability measure $\nu \in \mathcal{M}_{\mathcal{B}}$, by the previous lemma, we know that the moment-generating function of $S_n$ under the probability measure $\tilde{\mathbb{P}} \otimes \nu$ is given by $\int_X P^n_z (f) \, dm$  This leads us to understand the asymptotic behavior of the iterates of the Laplace operators $P_z$. Since they are smooth perturbations of the quasi-compact operator $P$, one can apply here the standard theory of perturbations for linear operators (see for instance theorem III.8 in \cite{HH}), and get the following :

\begin{lemma} \label{pert}

There exists $\epsilon_1 > 0$, $\eta_1>0$, $\eta_2 > 0$, and complex-analytic functions $\lambda(.)$, $h(.)$, $m(.)$, $Q(.)$, all defined on $\mathbb{D}_{\epsilon_1} = \{z \in \mathbb{C} \, / \, \left|z \right| < \epsilon_1 \}$, which take values respectively in $\mathbb{C}$, $\mathcal{B}$, $\mathcal{B}^{\star}$, $\mathcal{L}(\mathcal{B})$ and satisfying for all $z \in \mathbb{D}_{\epsilon_1}$ :

\begin{enumerate}

\item $ \lambda(0) = 1, h(0) = h, m(0) = m, Q(0) = Q$;

\item $P_z(f) = \lambda(z) \langle m(z), f \rangle h(z) + Q(z) f$ for all $f \in \mathcal{B}$;

\item $\langle m(z), h(z) \rangle = 1$;

\item $Q(z)h(z) = 0$ and $m(z)Q(z) = 0$;

\item $|\lambda(z)| > 1 - \eta_1$;

\item $\|Q(z)^n \| \le C (1 - \eta_1 - \eta_2)^n$.

\end{enumerate}
Furthermore, $\left| \langle m, Q(z)^n f \rangle \right| \le C |z| (1 - \eta_1 - \eta_2)^n \|f\|$ for all $f \in \mathcal{B}$ and $z \in \mathbb{D}_{\epsilon_1}$.

\end{lemma}

For all $n \ge 0$, we hence have $P_z^n(f) = \lambda(z)^n \langle m(z), f \rangle h(z) + Q(z)^n f$. The asymptotic behavior of $P_z^n$ is clearly intimately related to the behavior of the leading eigenvalue $\lambda(z)$ in a neighborhood of $0$. We have the following :

\begin{lemma}

The leading eigenvalue $\lambda(.)$ satisfies $\lambda'(0) = \int \varphi d\mu = 0$ and $\lambda''(0) = \sigma^2 \ge 0$.

\end{lemma}

\begin{proof}

By corollary III.11 in \cite{HH}, $\lambda'(0) = \langle m(0), P'(0)h(0) \rangle$. As $m(0) = m$, $h(0) = h = \frac{d \mu}{dm}$ and $P'(0) f = C_1(f) = P(\varphi f)$ for any $f \in \mathcal{B}$, since $P(z) = \sum_{n \ge 0} \frac{z^n}{n!}C_n$ with $C_n(f) = P(\varphi^n f)$, the formula for $\lambda'(0)$ reads as $$\lambda'(0) = \langle m, P(\varphi h) \rangle = \langle m, \varphi h \rangle = \int \varphi d \mu = 0.$$

Using again corollary III.11 in \cite{HH}, we have $$\lambda''(0) = \langle m(0), P''(0) h(0) \rangle + 2 \langle m(0), P'(0) \tilde{h} \rangle,$$ where $\tilde{h}$ is the unique element of $\mathcal{B}$ satisfying $\langle m(0), \tilde{h} \rangle = 0$ and $\left( \lambda(0) - P(0) \right) \tilde{h} = \left( P'(0) - \lambda'(0) \right) h(0)$.

This implies that $\tilde{h}$ is the unique element of $\mathcal{B}$ satisfying $\langle m, \tilde{h} \rangle = 0$ and $ (I - P) \tilde{h} = P(\varphi h)$. By corollary III.6 in \cite{HH}, $\tilde{h}$ is given by $\tilde{h} = \sum_{n\ge 0} Q(0)^n (\varphi h)$. But $Q(0)^n(\varphi h) = P^n(\varphi h) - \langle m, \varphi h \rangle h = P^n(\varphi h)$, since $\langle m, \varphi h \rangle = \int \varphi d \mu = 0$. Hence $\tilde{h} = \sum_{n \ge 0} P^n(P(\varphi h)) = \sum_{n \ge 1} P^n(\varphi h)$.

On one hand, we have $\langle m(0), P''(0) h(0) \rangle = \langle m, C_2(h) \rangle = \langle m, P( \varphi^2 h) \rangle = \int \varphi^2 d \mu.$ On the other hand, $$\langle m(0), P'(0) \tilde{h} \rangle = \langle m, P(\varphi \tilde{h}) \rangle = \langle m, \varphi \tilde{h} \rangle = \sum_{n \ge 1} \langle m, \varphi P^n(\varphi h) \rangle = \sum_{n \ge 1} \int \varphi P^n(\varphi h) dm = \sum_{n \ge 1} \int U^n \varphi \, \varphi d \mu.$$ Summing these two parts, we recognize the formula for $\sigma^2$ given by proposition \ref{green_kubo}. \qed
\end{proof}

Then, $\lambda(\frac{it}{\sqrt{n}})^n = (1 - \frac{\sigma^2 t^2}{2n} + o(\frac{1}{n}))^n$ goes to $e^{-\frac{\sigma^2 t^2}{2}}$, from which it follows that $\mathbb{E}_{\tilde{\mathbb{P}} \otimes \nu}(e^{i \frac{t}{\sqrt{n}} S_n}) = \lambda(\frac{it}{\sqrt{n}})^n \langle m(\frac{it}{\sqrt{n}}), f \rangle \langle m, h(\frac{it}{\sqrt{n}}) \rangle + \langle m, Q(\frac{it}{\sqrt{n}})f \rangle$ goes also to $e^{-\frac{\sigma^2 t^2}{2}}$, for each $t \in \mathbb{R}$, when $n \to \infty$, which it implies the CLT by L\'evy's continuity theorem. Remark that the previous identity holds for any measure $\nu \in \mathcal{M}_{\mathcal{B}}$ and their associated density $f$.

We can furthermore prove a rate of convergence in the CLT, when $\sigma^2 > 0$ :

\begin{lemma}  \label{speed_conv_carac}

There exists $C > 0$ and $\rho < 1$ such that for all $t \in \mathbb{R}$ and $n\ge 0$ with $\frac{\left| t \right|}{\sqrt{n}}$ sufficiently small, and all $\nu \in \mathcal{M}_{\mathcal{B}}$, we have $$\left| \mathbb{E}_{\tilde{\mathbb{P}} \otimes \nu}(e^{i \frac{t}{\sqrt{n}} S_n}) - e^{-\frac{1}{2} \sigma^2 t^2} \right| \le C \| \nu \| \left( e^{- \frac{\sigma^2 t^2}{2}} \left(\frac{|t | + |t|^3}{\sqrt{n}}\right) + \frac{ |t|}{\sqrt{n}} \rho^n \right).$$

\end{lemma}

\begin{proof}
This follows from the third order differentiability of $\lambda(.)$ at $0$ : for $\frac{t}{\sqrt{n}}$ small enough, $\lambda(\frac{it}{\sqrt{n}})^n = \left( 1 - \frac{\sigma^2 t^2}{2 n} + \mathcal{O}\left(\frac{|t|^3}{n\sqrt{n}}\right) \right)^n = e^{- \frac{\sigma^2 t^2}{2}} + \mathcal{O}\left( e^{-\frac{\sigma^2 t^2}{2}} \frac{|t|^3}{\sqrt{n}}\right)$. Recalling that $f = \frac{d \nu}{dm}$ and $\| \nu \| = \| f \| \ge C$, where the constant $C$ comes from the continuous embedding $\mathcal{B} \subset L^1(m)$ and is independent of $\nu$, we have $$\begin{aligned} &\mathbb{E}_{\tilde{\mathbb{P}} \otimes \nu}(e^{i \frac{t}{\sqrt{n}} S_n})& &=& &\lambda(\frac{it}{\sqrt{n}})^n \langle m(\frac{it}{\sqrt{n}}),f \rangle \langle m, h(\frac{it}{\sqrt{n}}) \rangle + \langle m, Q(\frac{it}{\sqrt{n}})^n f \rangle&  \\  && &=&  &\left( e^{- \frac{\sigma^2 t^2}{2}} + \mathcal{O}\left( e^{-\frac{\sigma^2 t^2}{2}} \frac{|t|^3}{\sqrt{n}}\right) \right) \left( 1 + \mathcal{O}\left( \frac{|t|}{\sqrt{n}} \|f \| \right) \right) + \mathcal{O}\left( \frac{|t|}{\sqrt{n}} \rho^n \| f \| \right),& \end{aligned}$$ where the first line follows from lemma \ref{fund} applied to $f$ and $z = \frac{it}{\sqrt{n}}$ and item 2 of lemma \ref{pert}. This implies the result. \qed
\end{proof}

From this lemma, we deduce that $\left| \mathbb{E}_{\tilde{\mathbb{P}} \otimes \nu}(e^{i \frac{t}{\sqrt{n}} S_n}) - e^{-\frac{1}{2} \sigma^2 t^2} \right| = \mathcal{O}\left( \frac{1 + |t|^3}{\sqrt{n}} \right)$, which will be useful latter, when proving a quenched CLT. The precise estimate of the lemma also implies a rate of convergence of order $\frac{1}{\sqrt{n}}$ in the CLT, using the Berry-Ess\'een inequality. We refer to \cite{HH} or \cite{Dur} for a scheme of proof :

\begin{theorem}

If $\sigma^2 > 0$, there exists $C >0$ such that for all $\nu \in \mathcal{M}_{\mathcal{B}}$ :  $$\sup_{t \in \mathbb{R}} \, \left| \tilde{\mathbb{P}} \otimes \nu \left( \frac{S_n}{\sqrt{n}} \le t \right) - \frac{1}{ \sigma \sqrt{2 \pi}} \int_{-\infty}^t e^{-\frac{u^2}{2 \sigma^2}} du \right| \le \frac{C \| \nu \|}{\sqrt{n}}.$$

\end{theorem}

We turn now to the proof of the LDP. For this, we will show the convergence of $\frac{1}{n} \log \mathbb{E}_{\tilde{\mathbb{P}} \otimes \nu}(e^{\theta S_n})$ for small enough $\theta \in \mathbb{R}$ and then apply Gartner-Ellis theorem \cite{DZ, El}. Proofs are a verbatim copy of those from \cite{AV}.

\begin{lemma}

There exists $0 < \epsilon_2 < \epsilon_1$ such that for every $\theta
\in \mathbb{R}$ with $|\theta| < \epsilon_2$, we have $\lambda(\theta)
> 0$. Furthermore, the functions $h(.)$ and $m(.)$ can be redefined in such a way that they still satisfy conclusions of lemma \ref{pert}, while they also verify $h(\theta) \ge 0$, $m(\theta) \ge 0$ for $\theta \in \mathbb{R}$.
\end{lemma}

\begin{proof}
As $P_{\theta}$ is a real operator, we have $P_{\theta} \overline{f} = \overline{P_{\theta}f}$ for all $f \in \mathcal{B}$. So, we have $P_{\theta} \overline{h(\theta)} = \overline{P_{\theta} h(\theta)} = \overline{\lambda(\theta)} \, \overline{h(\theta)}$. Since $\lambda(\theta)$ is the unique eigenvalue of $P_{\theta}$ with maximal modulus, we get $\overline{\lambda(\theta)} =
\lambda(\theta)$, and hence $\lambda(\theta) \in \mathbb{R}$. Since $\lambda(0) = 1$, by a continuity argument, we obtain $\lambda(\theta) > 0$ for small $\theta$. For $z \in \mathbb{C}$
small enough, $\langle m(z), \mathds{1} \rangle \neq 0$. We define $\tilde{h}(z) = \langle m(z), \mathds{1}\rangle h(z)$ and $\tilde{m}(z) = \langle m(z), \mathds{1}\rangle ^{-1} m(z)$. Those new eigenfunctions satisfy obviously the conclusions of the previous proposition. We have just to prove that $\tilde{h}(\theta)$ and $\tilde{m}(\theta)$ are positive for $\theta \in \mathbb{R}$ small enough. By the spectral decomposition of $P_{\theta}$, we see that $\lambda(\theta)^{-n}P_{\theta}^n \mathds{1}$ goes to $\tilde{h}(\theta)$ in $\mathcal{B}$, and hence in $L^1(m)$. We then get $\tilde{h}(\theta) \ge 0$ because $P_{\theta}$ is a positive operator and $\lambda(\theta)$ is positive too. Now, let $\psi(\theta) \in \mathcal{B}^{\star}$ positive such that $\langle \psi(\theta), \tilde{h}(\theta) \rangle = 1$ \footnote{Choose $\psi(\theta) = \alpha(\theta)^{-1}m$, where $\alpha(\theta) = \langle m, \tilde{h}(\theta) \rangle$ is positive, since $\tilde{h}(\theta)$ is a positive element of $L^1(m)$.}. Then, $\lambda(\theta)^{-n}(P_{\theta}^{\star})^n \psi(\theta)$ goes to $\langle \psi(\theta), h(\theta) \rangle m(\theta) = \tilde{m}(\theta)$, which proves that $\tilde{m}(\theta)$ is a positive linear form. \qed
\end{proof}

We denote $\Lambda(\theta) = \log \lambda(\theta)$. We then have

\begin{proposition}

For every $\nu \in \mathcal{M}_{\mathcal{B}}$, there exists $0 < \epsilon_3 < \epsilon_2$ such that for every $\theta
\in \mathbb{R}$ with $|\theta| < \epsilon_3$, we have $$\lim_{n \to \infty} \frac{1}{n} \log \mathbb{E}_{\tilde{\mathbb{P}} \otimes \nu}(e^{\theta S_n})  = \Lambda(\theta)$$
\end{proposition}

\begin{proof}
Let $f \in \mathcal{B}$ be the density $\frac{d\nu}{dm}$. We have the identity $$\begin{aligned} \mathbb{E}_{\tilde{\mathbb{P}} \otimes \nu}(e^{\theta S_n}) = \langle m, P_{\theta}^n(f)\rangle & = \lambda(\theta)^n \langle m(\theta), f\rangle \, \langle m, h(\theta) \rangle + \langle m, Q(\theta)^n f \rangle \\ & = \lambda(\theta)^n \left( \langle m(\theta), f \rangle \, \langle m, h(\theta) \rangle + \lambda(\theta)^{-n} \langle m, Q(\theta)^n f \rangle \right) \end{aligned}$$ All involved quantities are positive, hence we can write $$\frac{1}{n} \log \mathbb{E}_{\tilde{\mathbb{P}} \otimes \nu}(e^{\theta S_n}) = \log \lambda(\theta) +\frac{1}{n} \log \left( \langle m(\theta), f \rangle \, \langle m, h(\theta) \rangle + \lambda(\theta)^{-n} \langle m, Q(\theta)^n f \rangle \right)$$
Since $\lim_{\theta \to 0} \langle m(\theta), f\rangle \, \langle m, h(\theta) \rangle =
1$ and since the spectral radius of $Q(\theta)$ is strictly less than $\lambda(\theta)$, it's easy to see that for $\theta$ small enough, we have $$\lim_{n \to \infty} \frac{1}{n} \log \left(\langle m(\theta), f \rangle \, \langle m, h(\theta) \rangle + \lambda(\theta)^{-n} \langle m, Q(\theta)^n f \rangle \right) = 0.$$ \qed
\end{proof}

To complete the proof, it suffices to prove that $\Lambda$ is a differentiable function, strictly convex in a neighborhood of $0$, which is indeed the case since $\lambda(.)$ is complex-analytic and we have supposed $\Lambda''(0) = \lambda''(0) = \sigma^2 > 0$. A local version of the Gartner-Ellis theorem (a precise statement can be found e.g. in lemma XIII.2 in \cite{HH}) finishes the proof.

It is interesting to notice that the annealed LDP implies almost immediately a quenched upper bound, with the same rate function for almost every realization :

\begin{proposition}

For every $\nu \in \mathcal{M}_{\mathcal{B}}$, for every small enough $\epsilon > 0$ and for $\tilde{\mathbb{P}}$-almost every $\underline{\omega}$, we have $$\limsup_{n \to \infty} \frac{1}{n} \log \nu (\{ x \in X \, / \, S_n(\underline{\omega}, x) > n \epsilon \}) \le - c(\epsilon)$$

\end{proposition}

\begin{proof}
Let $\epsilon > 0$ small enough such that the annealed LDP holds. Let $0 < \gamma < 1$ and define $$A_n = \{ \underline{\omega} \in \tilde{\Omega} \, / \, \nu( \{ x \in X \, / \, S_n(\underline{\omega}, x) > n \epsilon \}) \ge e^{-n(1 -  \gamma) c(\epsilon)} \}.$$
By the annealed LDP, we have $\tilde{\mathbb{P}} \otimes \nu (S_n > n\epsilon) \le C e^{-n (1 - \frac{\gamma}{2}) c(\epsilon)}$ for some $C = C(\gamma, \epsilon)$, and hence Markov inequality yields $$\tilde{\mathbb{P}}(A_n) \le e^{n(1- \gamma)c(\epsilon)} \tilde{\mathbb{P}} \otimes \nu (S_n > n\epsilon) \le C e^{-n \frac{\gamma}{2} c(\epsilon)}.$$ By the Borel-Cantelli lemma, we have that $\tilde{\mathbb{P}}$-almost every $\underline{\omega}$ lies in finitely many $A_n$ whence $\limsup_{n \to \infty} \frac{1}{n} \log \nu (\{ x \in X \, / \, S_n(\underline{\omega}, x) > n \epsilon \}) \le -(1 - \gamma) c(\epsilon)$ for $\tilde{\mathbb{P}}$-almost every $\underline{\omega}$. As $\gamma$ can be a rational number arbitrarily close to $0$, we get $\limsup_{n \to \infty} \frac{1}{n} \log \nu( \{ x \in X \, / \, S_n(\underline{\omega}, x) > n \epsilon \}) \le - c(\epsilon)$  for $\tilde{\mathbb{P}}$-almost every $\underline{\omega}$.  \qed
\end{proof}

We can also prove a local limit theorem.

\begin{definition}
We will say that $\varphi$ is aperiodic if for all $t \neq 0$, the spectral radius of $P_{it}$ is strictly less than $1$.
\end{definition}

\begin{theorem}[Local Limit Theorem]
If $\sigma^2> 0$ and $\varphi$ is aperiodic, then, for all $\nu \in \mathcal{M}_{\mathcal{B}}$ and all bounded interval $I \subset \mathbb{R}$, $$\lim_{n \to \infty} \sup_{s \in  \mathbb{R}} \left|\sigma \sqrt{n} \, \tilde{\mathbb{P}} \otimes \nu (s + S_n \in I) - \frac{1}{\sqrt{2 \pi}} e^{-\frac{s^2}{2 n \sigma^2 }} |I| \right| = 0.$$

\end{theorem}

\begin{proof}
We follow the proof given by Breiman \cite{Brei} in the iid case. See also Rousseau-Egele \cite{R-E} for a proof in a dynamical context. By a density argument, it is sufficient to prove that, uniformly in $s \in \mathbb{R}$, $\left| \sigma \sqrt{n} \, \mathbb{E}_{\tilde{\mathbb{P}} \otimes \nu}(g(S_n + s)) - \frac{1}{\sqrt{2\pi}} e^{-\frac{s^2}{2 n \sigma^2}} \int_{\mathbb{R}} g(u)du \right|$ goes to $0$ as $n \to \infty$, for all $g \in L^1(\mathbb{R})$ for which the Fourier transform $\hat{g}$ is continuous with compact support.
Using Fourier's inversion formula, we first write $$\sigma \sqrt{n} \mathbb{E}_{\tilde{\mathbb{P}} \otimes \nu}(g(S_n + s)) = \frac{\sigma \sqrt{n}}{2\pi} \int_{\mathbb{R}} e^{its} \hat{g}(t) \left( \int_X P_{it}^n(f) \, dm \right) dt.$$
Let $\delta > 0$ be such that the support of $\hat{g}$ is included in $[-\delta, + \delta]$, and, remembering that $\lambda(it) = 1 - \frac{\sigma^2 t^2}{2} + o(t^2)$ and $\langle m(it), f \rangle \langle m, h(it) \rangle = 1 + \mathcal{O}(|t|)$, choose $0 < \tilde{\delta} < \delta$ small enough in such a way that $|\lambda(it)| \le 1 - \frac{\sigma^2 t^2}{4} \le e^{- \frac{t^2 \sigma^2}{4}}$ and $\left|\langle m(it), f \rangle \langle m, h(it) \rangle - 1 \right| \le C |t|$ for $|t| < \tilde{\delta}$.
Using $$\frac{1}{\sqrt{2\pi}} e^{-\frac{s^2}{2n \sigma^2}}\int_{\mathbb{R}} g(u)du = \frac{\hat{g}(0) \sigma}{2 \pi} \int_{\mathbb{R}} e^{\frac{its}{\sqrt{n}}} e^{- \frac{\sigma^2 t^2}{2}} dt$$ and $$\int_X P_{it}^n (f) \, dm = \lambda(it)^n \langle m, h(it) \rangle \langle m(it), f \rangle + \langle m, Q(it)^n f \rangle$$ for $|t| < \tilde{\delta}$, we can write

$$ \begin{aligned} &\sigma \sqrt{n} \, \mathbb{E}_{\tilde{\mathbb{P}} \otimes \nu}(g(S_n + s))& &-& &\frac{1}{\sqrt{2\pi}} e^{-\frac{s^2}{2 n \sigma^2}} \int_{\mathbb{R}} g(u)du& \\
&& &=& &\frac{\sigma}{2 \pi} \left( \int_{|t| < \tilde{\delta} \sqrt{n}} e^{\frac{its}{\sqrt{n}}} \left( \hat{g}(\frac{t}{\sqrt{n}}) \lambda(\frac{it}{\sqrt{n}})^n - \hat{g}(0) e^{- \frac{\sigma^2 t^2}{2}} \right) dt \right. & \\
&& &+& &\int_{|t| < \tilde{\delta}\sqrt{n}} e^{\frac{its}{\sqrt{n}}} \hat{g}(\frac{t}{\sqrt{n}}) \lambda(\frac{it}{\sqrt{n}})^n \left( \langle m(\frac{it}{\sqrt{n}}), f \rangle \langle m, h(\frac{it}{\sqrt{n}}) \rangle - 1 \right) dt& \\
&& &+& &\sqrt{n} \int_{|t| < \tilde{\delta}} e^{its} \hat{g}(t) \langle m, Q(it)^n f \rangle dt& \\
&& &+& & \left. \sqrt{n} \int_{\tilde{\delta} \le |t| \le \delta} e^{its} \hat{g}(t) \langle m, P_{it}^n f \rangle dt   - \hat{g}(0) \int_{|t| \ge \tilde{\delta} \sqrt{n}} e^{\frac{its}{\sqrt{n}}} e^{-\frac{\sigma^2 t^2}{2}} dt \right) & \\ && &=&  &\frac{\sigma}{2 \pi} \left( A_n^{(1)}(s) + A_n^{(2)}(s) + A_n^{(3)}(s) + A_n^{(4)}(s) + A_n^{(5)}(s) \right).&
\end{aligned}$$

One has $\left| A_n^{(1)} (s) \right| \le \int_{|t| < \tilde{\delta} \sqrt{n}} \left| \hat{g}(\frac{t}{\sqrt{n}}) \lambda(\frac{it}{\sqrt{n}})^n - \hat{g}(0) e^{- \frac{\sigma^2 t^2}{2}}\right| dt$, so $\sup_{s \in \mathbb{R}} \left| A_n^{(1)} (s) \right| \to 0$ by dominated convergence, since $\left| \hat{g}(\frac{t}{\sqrt{n}}) \lambda(\frac{it}{\sqrt{n}})^n - \hat{g}(0) e^{- \frac{\sigma^2 t^2}{2}} \right| \le \| \hat{g} \|_{\rm sup} \left( e^{- \frac{\sigma^2 t^2}{4}} + e^{- \frac{\sigma^2 t^2}{2}}\right)$ is integrable on $\mathbb{R}$ and $\hat{g}(\frac{t}{\sqrt{n}}) \lambda(\frac{it}{\sqrt{n}})^n \to \hat{g}(0) e^{- \frac{\sigma^2 t^2}{2}}$.

For the second term, we can bound it by $\frac{C \| \hat{g}  \|_{\rm sup}}{\sqrt{n}} \int_{|t| < \tilde{\delta} \sqrt{n}}|t| e^{- \frac{\sigma^2 t^2}{4}} dt = \mathcal{O}\left( \frac{1}{\sqrt{n}} \right)$ and so $\sup_{s \in \mathbb{R}} \left| A_n^{(2)} (s) \right| \to 0$

The third term is bounded by $C \sqrt{n} \| \hat{g}  \|_{\rm sup} \rho^n$, and so $\sup_{s \in \mathbb{R}} \left| A_n^{(3)} (s) \right| \to 0$.

By dominated convergence, we have clearly $\sup_{s \in \mathbb{R}} \left| A_n^{(5)} (s) \right| \to 0$, so it remains to deal with the fourth term. This is where the aperiodicity assumption plays a role. Denoting by $r(P_{it})$ the spectral radius of the operator $P_{it}$, we know that the u.s.c. function  $t \mapsto r(P_{it})$ reaches its maximum on the compact set $\{ \tilde{\delta} \le |t| \le \delta \}$, which is then $< 1$ by assumption. Since the set $\{P_{it} \}_{\tilde{\delta} \le |t| \le \delta}$ is bounded, there exists $C$ and $\theta < 1$ such that $\| P_{it}^n \| \le C \theta^n$ for all $\tilde{\delta} \le |t| \le \delta$ and all $n \ge 0$. Then one has $\sup_{s \in \mathbb{R}} \left| A_n^{(4)}(s) \right| \le C \sqrt{n} \theta^n \| \hat{g} \|_{\rm sup} \| m\|  \|f \| \to 0$, which concludes the proof. \qed
\end{proof}

We give a concrete criterion to check the aperiodicity assumption :

\begin{proposition}

Assume that the stationary measure $\mu$ is equivalent to $m$, and that the spectral radius (resp. the essential spectral radius) of $P_{it}$ is less (resp. strictly less) than $1$ for all $t \in \mathbb{R}$. If $\varphi$ is not aperiodic, then there exists $t  \neq 0$, $\lambda \in \mathbb{C}$ with $|\lambda| = 1$ and a measurable function $g : X \to \mathbb{C}$ such that $gh  \in \mathcal{B}$ and $\lambda g(T_{\omega} x) = e^{it \varphi(x)} g(x)$ for $m$-ae $x$ and $\mathbb{P}$-ae $\omega$.

\end{proposition}

\begin{proof}
Suppose that the spectral radius of $P_{it}$ is greater or equal than $1$ for some $t \neq 0$. By the assumptions on the spectral radius, this implies that there is an eigenvalue $\lambda$ of $P_{it}$ satisfying $| \lambda | =1$. Let $f \in \mathcal{B}$ a corresponding eigenvector, and define $g = \frac{f}{h}$. This definition makes sense $m$-ae, by the assumption on $\mu$. We then have $P(\phi g h) = gh$, where $\phi = \bar{\lambda} e^{it \varphi}$. We then lift this relation to the skew-product : by Lemma \ref{transfer_skew}, we have $P_S(\phi_{\pi} g_{\pi} h_{\pi}) = g_{\pi} h_{\pi}$, where $P_S$ is the transfer operator for the skew-product system, defined w.r.t. the measure $\tilde{\mathbb{P}} \otimes m$. See Section \ref{clt_martingale} for the notations. By Proposition 1.1 in Morita \cite{Mor3}, we deduce that $g_{\pi} \circ S = \phi_{\pi} g_{\pi}$, $\tilde{\mathbb{P}} \otimes m$- ae. We conclude the proof by writing explicitly this relation. \qed
\end{proof}

\begin{remark}

\begin{enumerate}

\item The assumptions on the spectral radius of $P_{it}$ are usually proved by mean of a Lasota-Yorke inequality for each $P_{it}$, which usually follow in the same way we prove a Lasota-Yorke inequality for the transfer operator $P$. See the works of Rousseau-Egele \cite{R-E}, Morita \cite{Mor3}, Broise \cite{Bro} or Aaronson-Denker-Sarig-Zweimuller \cite{ADSZ} for one-dimensional deterministic examples.

\item The previous Proposition shows that if $\varphi$ is not aperiodic for the random system, then it is not aperiodic for almost each deterministic system $T_{ \omega}$ in the usual sense, and that almost all aperiodicity equations share a common regular solution $g$. For instance, if the set $\Omega$ is finite and if we know that $\varphi$ is aperiodic for one map $T_{\omega}$, then it is aperiodic for the random system. This can be checked using known techniques, see \cite{AD, ADSZ, Bro} among many others for more details.

\end{enumerate}

\end{remark}

We conclude this section with an annealed vector-valued almost sure invariance principle. First recall the definition.

\begin{definition}
For $\lambda \in (0, \frac{1}{2}]$, and $\Sigma^2$ a (possibly degenerate) symmetric semi-positive-definite $d \times d$ matrix, we say that an $\mathbb{R}^d$-valued process $\left(X_n \right)_n$ satisfies an almost sure invariance principle (ASIP) with error exponent $\lambda$ and limiting covariance $\Sigma^2$ if there exist, on another probability space, two processes $\left( Y_n \right)_n$ and $\left( Z_n \right)_n$ such that :

\begin{enumerate}

\item the processes $\left(X_n \right)_{n}$ and $\left( Y_n \right)_n$ have the same distribution;

\item the random variables $Z_n$ are independent and distributed as $\mathcal{N}(0, \Sigma^2)$;

\item almost surely, $ \left| \sum_{k=0}^{n-1} Y_k - \sum_{k=0}^{n-1} Z_k \right| = o(n^{\lambda})$.

\end{enumerate}

\end{definition}

The ASIP has a lot of consequences, such as a functional central limit theorem, a law of the iterated logarithm, etc ... See Melbourne and Nicol \cite{MN} and references therein for more details.

Let $\varphi : X \to \mathbb{R}^d$ be a bounded vector-valued observable such that each component $\varphi_j : X \to \mathbb{R}$, $j=1, \ldots, d$, belongs to $\mathcal{B}_0$, with $\int_X \varphi_j \, d\mu = 0$.
Define as before $X_k(\underline{\omega}, x) = \varphi(T_{\underline{\omega}}^k x)$.

\begin{theorem}

The covariance matrix $\frac{1}{n}{\rm cov} \left(\sum_{k=0}^{n-1} X_k \right)$ converges to a matrix $\Sigma^2$ and the process $(X_n)_n$, defined on the probability space $(\tilde{\Omega} \times X, \tilde{\mathbb{P}} \otimes \mu)$, satisfies an ASIP with limiting covariance $\Sigma^2$, for any error exponent $\lambda > \frac{1}{4}$.

\end{theorem}

\begin{proof}
We will apply results from Gou\"{e}zel \cite{G10}. Namely, we construct a family of operators $(\mathcal{L}_t)_{t \in \mathbb{R}^d}$ acting on $\mathcal{B}$ which codes the characteristic function of the process $(X_n)_n$ and we check assumptions (I1) and (I2) from \cite{G10}. For $k \ge 0$ and $j_1, \ldots, j_k \in \{1, \ldots, d \}$, define $C_{j_1, \ldots, j_k}$ by $C_{j_1, \ldots, j_k}(f) = P(\varphi_{j_1} \ldots \varphi_{j_k} f)$. The assumptions on $\mathcal{B}_0$ and $\mathcal{B}$ show that $C_{j_1, \ldots, j_k}$ acts continuously on $\mathcal{B}$, with a norm bounded by $C^k \| \varphi_{j_1} \|_0 \ldots \| \varphi_{j_k} \|_0$. Now, for $t = (t_1, \ldots, t_d) \in \mathbb{R}^d$, define $$\mathcal{L}_t = \sum_{k=0}^{\infty} \frac{i^k}{k!} \sum_{j_1, \ldots, j_k = 1}^d t_{j_1} \ldots t_{j_k} C_{j_1, \ldots, j_k}.$$ This defines on $\mathbb{R}^d$ a real-analytic family of bounded operators on $\mathcal{B}$, since $$\sum_{k=0}^{\infty}  \| \frac{i^k}{k!} \sum_{j_1, \ldots, j_k = 1}^d t_{j_1} \ldots t_{j_k} C_{j_1, \ldots, j_k} \| \le \sum_{k=0}^{\infty} \frac{1}{k!} \sum_{j_1, \ldots, j_k = 1}^d |t_{j_1}| \ldots |t_{j_k}| \| C_{j_1, \ldots, j_k} \| \le e^{C \sum_{j=0}^d |t_j| \| \varphi_j \|_0} < \infty.$$
For $t \in \mathbb{R}^d$ and $f \in \mathcal{B}$, we have $\mathcal{L}_t(f) = P(e^{i \langle t, \varphi \rangle} f)$, and so the family $\{\mathcal{L}_t\}_{t \in \mathbb{R}^d}$ codes the characteristic function of the process $(X_n)_n$ in the sense of \cite{G10}, as easily seen using lemma \ref{fund}. Since $\mathcal{L}_0 = P$ has a spectral gap on $\mathcal{B}$, this implies (I1). To check (I2), we only need, by proposition 2.3 in \cite{G10}, the continuity of the map $t \mapsto \mathcal{L}_t$ at $t=0$, but this follows immediately from the real-analyticity of this map. \qed

\end{proof}

\section{Annealed central limit theorem via a martingale approximation} \label{clt_martingale}

The main goal of this section is to show that the classical martingale approach to the CLT, see Gordin \cite{Go} and Liverani \cite{Liv96}, can be easily adapted to the random setting, leading to a new proof of Theorem \ref{CLT_spec} for the stationary measure $\mu$, together with a generalization, in the next section, where a sequence of observables is considered, rather than a single one.

In this section, the annealed transfer operator $P$ and Koopman operator $U$ are defined by duality with respect to the stationary measure $\mu$, instead of the measure $m$. We assume moreover that we have decay of correlations for observables in $\mathcal{B}$ against $L^1(\mu)$, in the sense that $$\left| \int_X f \, U^n g \, d\mu - \int_X f \, d\mu \int_X g \, d \mu \right| \le C \lambda^n \|f \| \|g \|_{L^1_{\mu}}$$ for all $f \in \mathcal{B}$ and $g \in L^1(\mu)$. This is the case for instance if we assume that the density $h$ of the stationary measure is bounded away from $0$ and that $\mathcal{B}$ is continuously embedded in $L^1(m)$, see Proposition \ref{decay}.

Recall that we have the Markov operator $U$ which acts on functions defined on $X$ by $Uf(x) =  \int_{\Omega} f(T_{\omega} x) \, d\mathbb{P}(\omega)$. To $U$ is associated a transition probability on $X$ defined by $U(x,A) = U(\mathds{1}_A)(x) = \mathbb{P}( \{ \omega \, / \, T_{\omega}x \in A\})$. Recall also that the stationary measure $\mu$ satisfies $\mu U = \mu$, by definition. We can then define the canonical Markov chain associated to $\mu$ and $U$ :

Let $\Omega^{\star} = X^{\mathbb{N}_0} = \{ \underline{x} = (x_0, x_1, x_2, \ldots, x_n, \ldots ) \}$, endowed with the $\sigma$-algebra $\mathcal{F}$ generated by cylinder sets. As $X$ is Polish, this is also the Borel $\sigma$-algebra associated with the product topology.
The Ionescu-Tulcea's theorem asserts there exists an unique probability measure $\mu_c$ on $\Omega^{\star}$ such that $$\int_{\Omega^{\star}} f(\underline{x}) \, d\mu_c(\underline{x}) = \int_X \mu(dx_0) \int_X U(x_0, dx_1) \ldots \int_X U(x_{n-1}, dx_n) f(x_0, \ldots, x_n)$$ for every $n$ and every bounded measurable function $f : \Omega^{\star} \to \mathbb{R}$ which depends only on $x_0, \ldots, x_n$. If we still denote by $x_n$ the map which associates to each $\underline{x}$ its $n$-th coordinate $x_n$, then $\{x_n\}_{n \ge 0}$ is a Markov chain defined on the probability space $(\Omega^{\star},  \mathcal{F}, \mu_c)$, with initial distribution $\mu$, and transition probability $U$. By stationarity of the measure $\mu$, each $x_n$ is distributed accordingly to $\mu$.

We can define an unilateral shift $\tau$ on $\Omega^{\star}$. By stationarity, it preserves $\mu_c$. Recall also that we have a skew-product system $S$ on $\Omega^{\mathbb{N}} \times X$, defined by $S(\underline{\omega}, x) = (\sigma \underline{\omega}, T_{\omega_1}x)$, where $\sigma$ is the unilateral shift on $\Omega^{\mathbb{N}}$. $S$ preserves the probability measure $\tilde{\mathbb{P}} \otimes \mu = \mathbb{P}^{\otimes \mathbb{N}} \otimes \mu$. This system is related to the shift on $\Omega^{\star}$ in the following way :

Define $\Phi : \Omega^{\mathbb{N}} \times X \to \Omega^{\star}$ by $\Phi(\underline{\omega}, x) = (x, T_{\omega_1}x, T_{\omega_2} T_{\omega_1} x, \ldots, T_{\omega_n} \ldots T_{\omega_1} x, \ldots) = \{p(S^n(\underline{\omega}, x))\}_{n \ge 0}$, where $p(\underline{\omega}, x) = x$. We have the following :

\begin{lemma}

$\Phi$ is measurable, sends $\tilde{\mathbb{P}} \otimes \mu$ on $\mu_c$, and satisfies $\Phi \circ S = \tau \circ \Phi$.

\end{lemma}

\begin{proof}
The only non-trivial thing to prove is that $\tilde{\mathbb{P}}\otimes \mu$ is sent on $\mu_c$. For this, it is sufficient to prove that \begin{multline*} \int_X \mu(dx) \int_{\Omega^{\mathbb{N}}} f_0(x) f_1(T_{\underline{\omega}}^1x) \ldots f_n(T_{\underline{\omega}}^n x) \, d \tilde{\mathbb{P}}(\underline{\omega}) \, d\mu(x) \\ = \int_X \mu(dx_0) f_0(x_0) \int_X U(x_0, dx_1) f_1(x_ 1) \ldots \int_X U(x_{n-1}, dx_n) f_n(x_n) \end{multline*} for all $n \ge 0$ and all bounded measurable functions $f_0, \ldots, f_n : X \to \mathbb{R}$. We proceed by induction on $n$, the case $n = 0$ being obvious. We have $$ \scriptscriptstyle{ \begin{aligned} && &\int_X \mu(dx_0) f_0(x_0) \int_X U(x_0, dx_1) f_1(x_ 1) \ldots \int_X U(x_{n-1}, dx_n) f_n(x_n) \int_X U(x_n, dx_{n+1}) f_{n+1}(x_{n+1})& \\ &=&  &\int_X \mu(dx_0) f_0(x_0) \int_X U(x_0, dx_1) f_1(x_ 1) \ldots \int_X U(x_{n-1}, dx_n) f_n(x_n) Uf_{n+1}(x_n)& \\ &=&  &\int_X \mu(dx) \int_{\Omega^{\mathbb{N}}} f_0(x) f_1(T_{\underline{\omega}}^1 x) \ldots f_n(T_{\underline{\omega}}^n x) Uf_{n+1}(T_{\underline{\omega}}^n x)\, d \tilde{\mathbb{P}}(\underline{\omega}) \, d\mu(x)& \\ &=& &\int_X \mu(dx) \int_{\Omega^{\mathbb{N}}} f_0(x) f_1(T_{\underline{\omega}}^1 x) \ldots f_n(T_{\underline{\omega}}^n x) f_{n+1}(T_{\underline{\omega}}^{n+1} x) \, d \tilde{\mathbb{P}}(\underline{\omega}) \, d\mu(x)& \end{aligned}}$$ which concludes the proof. \qed
\end{proof}

Let $\pi_n : \Omega^{\star} \to X$ be the projection operator $\pi_n(x_0,  \ldots, x_n,  \ldots) = x_n$ and $\pi = \pi_0$ We lift each $\phi : X \to \mathbb{R}$ on $\Omega^{\star}$ by $\phi_{\pi} = \phi \circ \pi$. We then have $\mathbb{E}_{\mu}(\phi) = \mathbb{E}_{\mu_c}(\phi_{\pi})$.  Notice that $\pi_n = \pi \circ \tau^n$ and $p \circ S^n = \pi_n \circ \Phi$.

For a fixed observable $\phi : X \to \mathbb{R}$ with zero $\mu$-mean, define $X_k = \phi \circ p \circ S^k$ and $S_n = \sum_{k=0}^{n -1} X_k$. We have $X_k = \phi \circ \pi_k \circ \Phi =  \phi_{\pi} \circ \tau^k \circ \Phi$. Hence $S_n = (\sum_{k=0}^{n-1} \phi_{\pi}  \circ \tau^k) \circ \Phi$, and so the law of $S_n$ under $\tilde{\mathbb{P}} \otimes \mu$ is the law of the $n$-th Birkhoff sum of $\phi_{\pi}$ under $\mu_c$. So, to prove the CLT for $S_n$ under the probability measure $\tilde{\mathbb{P}} \otimes \mu$, it suffices to prove it for the Birkhoff sum of the observable $\phi_{\pi}$ for the symbolic system $(\Omega^{\star}, \tau, \mu_c)$.

To this end, we introduce the Koopman operator $\tilde{U}$ and the transfer operator $\tilde{P}$ associated to $(\Omega^{\star}, \tau, \mu_c)$. Since this system is measure-preserving, those operators satisfy $\tilde{P}^k \tilde{U}^k f = f$ and $\tilde{U}^k \tilde{P}^k f = \mathbb{E}_{\mu_c}(f | \mathcal{F}_k)$ for every $\mu_c$ integrable $f$, where $ \mathcal{F}_k = \tau^{-k} \mathcal{F} = \sigma(x_k, x_{k+1}, \ldots)$

We have the following
\begin{lemma} \label{Pphi}
For every $\phi : X \to \mathbb{R}$, we have $\tilde{P}( \phi_{\pi}) = (P \phi)_{\pi}$.
\end{lemma}

We will deduce this result from the corresponding statement for the transfer operator $P_S$ of the skew product :

\begin{lemma} \label{transfer_skew}
For every $\phi : X \to \mathbb{R}$, we have $P_S( \phi \circ p) = (P\phi) \circ p$.
\end{lemma}

\begin{proof}
Let $\psi : \Omega^{\mathbb{N}} \times X \to \mathbb{R}$ be an arbitrary element of $L^{\infty}(\tilde{\mathbb{P}} \otimes \mu)$. We have to show that $\int_{\Omega^{\mathbb{N}} \times X} (\phi \circ p) (\psi \circ S) \, d(\tilde{\mathbb{P}} \otimes \mu) = \int_{\Omega^{\mathbb{N}} \times X} ((P\phi) \circ p) \psi \, d(\tilde{\mathbb{P}} \otimes \mu)$.
But $$ \begin{aligned}
& \int_{\Omega^{\mathbb{N}} \times X} (\phi \circ p) (\psi \circ S) \, d(\tilde{\mathbb{P}} \otimes \mu)  = \int_{\Omega^{\mathbb{N}}} \int_X \phi(x) \psi(\sigma \underline{w}, T_{\omega_1} x) \,  d\mu(x) d \tilde{\mathbb{P}}(\underline{\omega}) & \\ &= \int_{\Omega^{\mathbb{N}}} \int_X P_{\omega_1} \phi(x) \psi(\sigma \underline{\omega}, x) \, d\mu(x) d \tilde{\mathbb{P}}(\underline{\omega})& \\  & =  \int_X \int_{\Omega^{\mathbb{N}}} \int_{\Omega} P_{\omega_1} \phi(x) \psi((\omega_2, \omega_3, \ldots), x) \, d\mathbb{P}(\omega_1) d \tilde{\mathbb{P}}(\omega_2, \omega_3, \ldots) d\mu(x)& \\ &=  \int_X P\phi(x) \int_{\Omega^{\mathbb{N}}} \psi(\underline{\omega}, x) \, d\tilde{\mathbb{P}} (\underline{\omega}) d\mu(x) = \int_{\Omega^{\mathbb{N}} \times X} ((P\phi) \circ p) \psi \, d(\tilde{\mathbb{P}} \otimes \mu)&
\end{aligned}$$ \qed

\end{proof}

\begin{proof}[Proof of lemma \ref{Pphi}] Let $\psi : \Omega^{\star} \to \mathbb{R}$ be an arbitrary element of $L^{\infty}(\mu_c)$. We have to show that $\int_{\Omega^{\star}} \phi_{\pi} (\psi \circ \tau) \, d\mu_c = \int_{\Omega^{\star}} (P\phi)_{\pi} \psi \, d\mu_c$. We have
$$\begin{aligned}
& \int_{\Omega^{\star}} \phi_{\pi} (\psi \circ \tau) \, d\mu_c =  \int_{\Omega^{\mathbb{N}} \times X} (\phi \circ \pi \circ \Phi) (\psi \circ \tau \circ \Phi) \, d(\tilde{\mathbb{P}} \otimes \mu)& \\ &= \int_{\Omega^{\mathbb{N}} \times X} (\phi \circ p) ( \psi \circ \Phi \circ S) \, d (\tilde{\mathbb{P}} \otimes \mu) =\int_{\Omega^{\mathbb{N}} \times X} P_S(\phi \circ p) (\psi \circ \Phi) \, d(\tilde{\mathbb{P}} \otimes \mu) & \\  &= \int_{\Omega^{\mathbb{N}} \times X} (P \phi \circ p) ( \psi \circ \Phi) \,  d(\tilde{\mathbb{P}} \otimes \mu)   = \int_{\Omega^{\mathbb{N}} \times X} (P\phi \circ \pi \circ \Phi) (\psi \circ \Phi) \, d(\tilde{\mathbb{P}} \otimes \mu) & \\ &= \int_{\Omega^{\star}} (P\phi)_{\pi} \psi \, d \mu_c &
\end{aligned} $$ \qed

\end{proof}

This helps us to construct a martingale approximation for the Birkhoff sums of $\phi_{\pi}$. We first remark upon the stationary case. By our assumption on decay of correlations, the series $w = \sum_{n=1}^{\infty} P^n \phi$ is convergent in $L^{\infty}(\mu)$ if $\phi \in \mathcal{B}$. We define $\chi = \phi_{\pi} + w_{\pi} - w_{\pi} \circ \tau$ on $\Omega^{\star}$. $\chi$ satisfies $\tilde{P} \chi = \tilde{P}(\phi_{\pi}) + \tilde{P}(w_{\pi}) - \tilde{P} \tilde{U} w_{\pi} = (P\phi)_{\pi} + (P w)_{\pi} - w_{\pi} = 0$, since $Pw =  w - P\phi $.

We claim that $\{\chi \circ \tau^k\}_{k\ge 0}$ is  a reverse martingale difference with respect to the decreasing filtration $\{\mathcal{F}_k\}_{k \ge 0}$. Indeed, we have $\mathbb{E}_{\mu_c}(\chi \circ \tau^k | \mathcal{F}_{k+1}) = \tilde{U}^{k+1} \tilde{P}^{k+1} \tilde{U}^k \chi = \tilde{U}^{k+1} \tilde{P} \chi = 0$. Uniqueness of the stationary measure also yields that the associated martingale is ergodic, and hence satisfies a CLT (see Billingsley \cite{Billingsley}).

Since $\sum_{k=0}^{n-1} \phi_{\pi} \circ \tau^k = \sum_{k=0}^{n-1} \chi \circ \tau^k + w_{\pi} \circ \tau^n - w_{\pi}$, and $\frac{w_{\pi} \circ \tau^n - w_{\pi}}{\sqrt{n}}$ goes to zero in probability (because it goes to $0$ in the $L^1$ norm), it follows that
$\frac{1}{\sqrt{n}}  S_n $ converges to the gaussian law $\mathcal{N}(0,\sigma^2)$ in distribution , where $\sigma^2 = \mathbb{E}_{\mu_c}(\chi^2)$, since  $\sum_{j=0}^{n-1} \frac{1}{\sqrt{n}} \chi \circ \tau^j $ converges to $\mathcal{N}(0,\sigma^2)$ in distribution.

\section{Dynamical Borel-Cantelli lemmas} \label{borel_cantelli}

In this section, we make the same assumptions as in the previous one. Recall the following result from \cite{Sprindzuk} :

\begin{theorem}
Let $f_k$ be a sequence of non-negative measurable functions on a measure space $(Y, \nu)$, and let $\bar{f_k}$, $\varphi_k$ be sequences of real numbers such that $0 \le \bar{f_k} \le \varphi_k \le M$ for all $k \ge 1$ and some $M > 0$. Suppose that $$\int_Y \left( \sum_{m < k \le n} \left( f_k(y) - \bar{f_k} \right) \right)^2 \, d\nu(y) \le C \sum_{m < k \le n} \varphi_k$$ for arbitrary integers $m < n$ and some $C> 0$. Then $$\sum_{1 \le k \le n} f_k(y) = \sum_{1 \le k \le n} \bar{f_k} + O(\Phi^{1/2}(n) \log^{3/2 + \epsilon} \Phi(n))$$ for $\nu$-a.e. $y \in Y$, for all $\epsilon > 0$ and $\Phi(n) = \sum_{1 \le k \le n} \varphi_k$.
\end{theorem}

Applying this result to the probability space $(\tilde{\Omega} \times X, \tilde{\mathbb{P}} \otimes \mu)$, and using decay of correlations, we get :

\begin{proposition}
If $\phi_n$ is a sequence of non-negative functions in $\mathcal{B}$, with $\sup_{n} \, \| \phi_n\| < \infty$ and
$E_n \to \infty$, where $E_n=\sum_{j=0}^{n-1}\int \phi_n \, d\mu$, then
\[
\lim_{n\to\infty} \frac{1}{E_n} \sum_{j=0}^{n-1} \phi_j (S^j (x,\omega))\to 1
\]
for $\tilde{\mathbb{P}} \otimes \mu$ a.e. $(\omega, x) \in \tilde{\Omega} \times X$.

\end{proposition}

See theorem 2.1 in Kim \cite{Kim} for a completely analogue proof in a deterministic setting. The annealed version of the Strong Borel-Cantelli property clearly implies a quenched version, namely for $\tilde{\mathbb{P}} $-a.e. $\omega$ for $\mu$-a.e. $x\in X$,
\[
\lim_{n\to\infty} \frac{1}{E_n} \sum_{j=0}^{n-1} \phi_j (S^j (x,\omega))\to 1
\]

We now show how to prove a CLT, following our martingale approach described in the previous section.

\subsection*{CLT for  Borel-Cantelli sequences}

 Let $p\in X$ and let $B_n(p)$ be a sequence of nested balls
about $p$ such that for  $0\le \gamma_2 \le \gamma_1 \le 1$ and constants $C_1$,$C_2>0$ we have $\frac{C_1}{n^{\gamma_1}} \le
 \mu (B_n(p))\le \frac{C_2}{n^{\gamma_2}}$.

 Let $\phi_n = \mathds{1}_{B_n(p)}$
be the characteristic function of $B_n (p)$. We assume that $\phi_n$ is a bounded sequences in $\mathcal{B}$, which is clearly the case when $\mathcal{B}$ is $\rm BV$ or Quasi-H\"{o}lder. We will sometimes write $\mathbb{E}[\phi ]$ or $\int \phi$  for the integral $\int \phi~d\mu$
when the context is understood.

First we lift $\phi_i$ to $\Omega^*$ and define $(\phi_i)_{\pi}=\phi_i\circ \pi$ and then we normalize and write $\tilde{\phi}_j=(\phi_j)_{\pi}-\int (\phi_j)_{\pi} d\mu_c$.

We are almost in the setting of~\cite[Proposition 5.1]{HNVZ} which states,

\begin{proposition}
Suppose $(T, X, \mu)$ is an ergodic transformation whose transfer operator $P$ satisfies, for some constants $C>0$, $0<\theta<1$,
\[
\|P^n \phi \|_{\mathcal{B}} \le C \theta^n \|\phi \|_{\mathcal{B}}
\]
 for all $\phi $ such that $\phi~d\mu=0$. Let $B_i:=B_i(p)$ be nested balls about a point $p$ such that for constants $0\le \gamma_2 \le \gamma_1 \le 1$,$C_1>0$,$C_2>0$ we have $\frac{C_1}{n^{\gamma_1} }\le \mu (B_n(p))\le \frac{C_2}{n^{\gamma_2}}$.  Let
 \[
 a_n^2:= E(\sum_{j=1}^n (1_{B_i}\circ T^i -\mu (B_i))^2
 \]
 Then
 \[
 \liminf \frac{a_n^2}{E_n}\ge 1
 \]
 and
 \[
 \frac{1}{a_n} \sum_{j=1}^n (\phi_j-\int \phi_j~d\mu) \circ T^j \to \mathcal{N}(0,1)
 \]

\end{proposition}

As a fairly direct corollary we may  show in our setting:

\begin{corollary}
\[
\frac{1}{a_n} \sum_{j=1}^n \tilde{\phi}_j \circ \tau^j \to \mathcal{N}(0,1)
\]
and hence
\[
 \frac{1}{a_n} \sum_{j=1}^n (\phi_j-\int \phi_j~d\mu)\circ S^j \to \mathcal{N}(0,1)
 \]

\end{corollary}

\begin{proof}

Define $\phi_0=1$ and for $n\ge 1$
\[
w_n =  P \phi_{n-1}+P^2 \phi_{n-2}+\ldots P^n \phi_0
\]
so that $w_1=P\phi_0$, $w_2=P\phi_1 +P^2\phi_0$, $w_3=P\phi_2+P^2\phi_1+P^3\phi_0$ etc...
For $n\ge 1$ define
\[
\psi_n= (\phi_n)_{\pi} - (w_{n+1})_{\pi} \circ \tau + (w_n)_{\pi}
\]

An easy calculation shows that $\tilde{P}\psi_n=0$ and hence $X_{ni}:=\psi_i\circ \tau^i/(a_n)$ is a  reverse martingale difference array with respect to the filtration
$\mathcal{F}_i$.

We have exponential decay of correlations in the sense that  if $j>i$ then
\begin{eqnarray*}
\left|\int \tilde{\phi}_i \circ \tau^i \tilde{\phi}_j \circ \tau_j d\mu_c \right| &=&\left|\int \tilde{\phi}_i  \tilde{\phi}_j \circ \tau^{j-i} d\mu_c  \right|\\
&\le& C\theta^{j-i} \left\| \phi_j-\int \phi_j~d\mu\right\|_{\mathcal{B}} \|\tilde{\phi}_j\|_1
\end{eqnarray*}
where $ \| \phi_j-\int \phi_j~d\mu\|_{\mathcal{B}}$ is bounded uniformly over $j$.

The proof of~\cite[Proposition 5.1]{HNVZ} may now be followed exactly to establish conditions $(a)$, $(b)$, $(c)$ and $(d)$ of Theorem 3.2 from Hall and Heyde~\cite{Hall_Heyde} as well as show that  the variance $a_n$  is unbounded. \qed \end{proof}

\section{Erd\"os-R\'enyi laws} \label{erdos_renyi}

Erd\"{o }s-R\'enyi limit laws give information on the maximal average gain precisely in the case where the length of the time window ensures  there is a non-degenerate limit. Recall the following proposition from \cite{DN} :

\begin{proposition} \label{prop:erdos1}

Let $(X,T, \mu)$ be a probability preserving transformation, and $\varphi : X \to \mathbb{R}$ be a mean-zero $\mu$-integrable function. Let $S_n(\varphi) = \varphi + \ldots + \varphi \circ T^{n-1}$.

\begin{enumerate}

\item Suppose that $\varphi$ satisfies a large deviation principle with rate  function $I$ defined on the open set $U$. Let $\alpha >0$ and set
$$ l_n=l_n(\alpha)=\left[\frac{\log n}{I(\alpha)}\right]\qquad n\in\mathbb N.$$
Then the upper Erd\"os-R\'enyi law holds, that is, for $\mu$ a.e. $x\in X$
$$ \limsup_{n\to\infty} \max\{S_{l_n} (\varphi) \circ T^j (x)/l_n: 0\le j\le n-l_n\} \le \alpha.$$

\item Suppose that for every $\epsilon>0$ the series $\sum_{n>0} \mu (B_n (\epsilon))$, where $B_n(\epsilon)=\{\max_{0\le m\le n-l_n} S_{l_n}\circ T^m \le l_n(\alpha-\epsilon)\}$ is summable.

Then the lower Erd\"os-R\'enyi law holds, that is, for  $\mu$ a.e. $x\in X$
$$ \liminf_{n\to\infty} \max\{S_{l_n} (\varphi) \circ T^j (x)/l_n: 0\le j\le n-l_n\} \ge \alpha.$$

\end{enumerate}

\end{proposition}

\begin{remark}
Assumptions (a) and (b) of Proposition \ref{prop:erdos1} together imply  that
\[
\lim_{n\to \infty}  \max_{0\le m\le n-l_n} \frac{S_{l_n}\circ T^m}{l_n}=\alpha.
\]
\end{remark}

In this section, we will suppose that $X = [0,1]$, and that the Banach space $\mathcal{B}$ is ${\rm BV}$, the space of functions of bounded variation on $[0,1]$. All maps $T_{\omega}$ are piecewise $C^2$, and we assume an uniform upper bound $L>0$ for their derivatives. We will apply the previous proposition to the symbolic system $(\Omega^{\star}, \tau, \mu_c)$ introduced before.

\begin{theorem}
Suppose $\phi: X\to \R$ is of bounded variation with $\int_X \phi \, d\mu = 0$ and define $S_n=\sum_{j=0}^{n-1} \phi_{\pi}\circ \tau^j$.
Let $\alpha >0$ and set
$$ l_n=l_n(\alpha)=\left[\frac{\log n}{I(\alpha)}\right]\qquad n\in\mathbb N$$ where $I(.)$ is the rate function associated to $\phi$, which exists by Theorem \ref{LDP_spec}.
Then
\[
\lim_{n\to \infty}  \max_{0\le m\le n-l_n} \frac{S_{l_n}\circ T^m}{l_n}=\alpha.
\]
\end{theorem}

\begin{proof}
Since $\phi$ satisfies an annealed LDP, which can be immediately lifted to a LDP for $\phi_{\pi}$, we need only prove $(2)$. As in the section on the logistic map in \cite{DN} we use a blocking argument to establish $(2)$.

For all $s > 0$, define $A_n^s = \{S_{l_n} \le l_n(\alpha - s )\}$. We fix $\epsilon > 0$, and consider $A_n^{\epsilon}$ and $A_n^{\epsilon / 2}$. Let $0<\eta<1$. We define $\varphi_{\epsilon}$ to be a Lipschitz function with Lipschitz norm at most $L^{(1+\eta)l_n}$ satisfying $\mathds{1}_{A_n^{\epsilon}} \le \varphi_{\epsilon} \le 1$ and $\mu(A_n^{\epsilon}) < \int_X \varphi_{\epsilon} \, d\mu < \mu(A_n^{\epsilon / 2})$, in the same way as in the proof of Theorem 3.1 in \cite{DN}.

Define $C_m (\epsilon)=\{S_{l_n}\circ \tau^m \le l_n(\alpha-\epsilon)\} $ and $B_n(\epsilon)=\bigcap_{m=0}^{n-l_n} C_m (\epsilon)$.
We use a blocking argument to take advantage of decay of correlations and
intercalate by blocks of length $(\log n)^{\kappa}$, $\kappa>1$. We define
\[
  E^0_n(\epsilon) := \bigcap_{m=0}^{[(n-(\log n)^{\kappa})/(\log n)^{\kappa})]} C_{m[(\log n)^{\kappa}]}(\epsilon)
  \]
and  in general for $0\le j < [\frac{n}{(\log n)^{\kappa}}]$
\[
E_n^j (\epsilon):=\bigcap_{m=0}^{[(n-(j+1)(\log n)^{\kappa})/(\log n)^{\kappa})]} C_{m[(\log n)^{\kappa}]}(\epsilon).
\]
Note that
$\mu (B_n (\epsilon) )\le \mu (E_n^0 (\epsilon) )$.  For each $j$, let
 $\psi_{j}  = \mathds{1}_{E_n^j (\epsilon)}$ denote  the characteristic function of $E_n^j (\epsilon)$.

By decay of correlations we have
\begin{eqnarray*}
\mu  (E_n^0 (\epsilon) ) &\le& \int \varphi_{\epsilon}\cdot  \psi_1 \circ \tau^{[(\log n)^{\kappa}]}d\mu_c\\
&\le&C \theta^{(\log n)^{\kappa}} \|\varphi_{\epsilon}\|_{BV} \|\psi_1\|_{1} +  \int \varphi_{\epsilon}~d\mu_c \int \psi_1~d\mu_c \\
&\le&  \int \varphi_{\epsilon}~d\mu_c  \int \psi_1 ~d\mu_c +C \theta^{(\log n)^{\kappa}} (L^{(1+\eta) l_n}).
\end{eqnarray*}
Applying  decay of correlations again to $  \int \psi_1 ~d\mu_c$ we iterate
and conclude
\[
\mu  (E_n^0 (\epsilon) ) \le n C \theta^{(\log n)^{\kappa}} L^{(1+\eta) l_n} +\mu( A_n^{\epsilon/2})^{n/(\log n)^{\kappa}}.
\]
The term $n C \theta^{(\log n)^{\kappa}} L^{(1+\eta) l_n}$ is clearly summable since $\kappa>1$. \qed

\end{proof}

\begin{remark}
We would obtain a quenched Erd\"os-R\'enyi law as well, if we could establish exponential decay of correlations for
 $ \tilde{\mathbb{P}}$ almost every $\underline{\omega}$, together with a quenched LDP for functions of bounded variation.
\end{remark}

\section{Quenched CLT for random one dimensional systems} \label{quenched_clt}

In this section, we consider the quenched CLT, that is a CLT holding for almost every fixed sequence $\underline{\omega}$. We first state a general result, which is basically a consequence of \cite{ALS}.
Let $\{T_{\omega}\}_{\omega \in \Omega}$ be a iid random dynamical system acting on $X$, with a stationary measure $\mu$. Let $\varphi : X \to \mathbb{R}$ be an observable with $\int_X  \varphi d \mu = 0$, and define as before $X_k(\underline{\omega}, x) = \varphi (T_{\underline{\omega}}^k x)$ and $S_n = \sum_{k=0}^{n-1} X_k$.
We will need to introduce a auxiliary random system defined as follows : the underlying probability space is still $(\Omega, \mathbb{P})$, while the auxiliary system acts on $X^2$, with associated maps $\hat{T}_{\omega}$ given by $\hat{T}_{\omega}(x,y) = (T_{\omega}x, T_{\omega}y)$. Define then a new observable $\hat{\varphi} :X^2 \to \mathbb{R}$ by $\hat{\varphi}(x,y) = \varphi(x) - \varphi(y)$, and denote its associated Birkhoff sums by $\hat{S}_n$.

\begin{theorem}
Assume there exists $\sigma^2 >0$ and a constant $C > 0$ such that for all $t \in \mathbb{R}$ and $n \ge 1$ with $\frac{t}{\sqrt{n}}$ small enough :

\begin{enumerate}

\item[(1)] $\left| \mathbb{E}_{\tilde{\mathbb{P}} \otimes \mu}(e^{i \frac{t}{\sqrt{n}} S_n}) - e^{- \frac{t^2 \sigma^2}{2}} \right| \le C \frac{1 + |t|^3}{\sqrt{n}}$,

\item[(2)] $\left| \mathbb{E}_{\tilde{\mathbb{P}} \otimes (\mu \otimes \mu)}(e^{i \frac{t}{\sqrt{n}}  \hat{S}_n}) - e^{- t^2 \sigma^2} \right| \le C \frac{1 + |t|^3}{\sqrt{n}}$.

\end{enumerate}

Suppose also that for $n \ge 1$ and $\epsilon > 0$ :

\begin{enumerate}

\item[(3)] $\tilde{\mathbb{P}} \otimes \mu \left( \left| \frac{S_n}{n} \right| \ge \epsilon \right) \le C e^{-C \epsilon^2 n}$.

\end{enumerate}

Then, the quenched CLT holds : for $\tilde{\mathbb{P}}$-a.e. sequence $\underline{\omega} \in \Omega^{\mathbb{N}}$ one has $$\frac{\sum_{k=0}^{n-1} \varphi \circ T_{\underline{\omega}}^k}{\sqrt{n}} \Longrightarrow_{\mu} \mathcal{N}(0, \sigma^2).$$

\end{theorem}

The first and third assumptions can be proved using the spectral approach described in this paper. Indeed the first one corresponds to lemma \ref{speed_conv_carac}, while the third follows from the LDP. To prove the second assumption, one must employ again the spectral technique, but with the auxiliary system introduced above and the observable $\hat{\varphi}$. There are mainly two difficulties : the obvious one is that the auxiliary system acts on a space whose dimension is twice the dimension of $X$, so that we have to use more complicated functional spaces. The other difficulty, less apparent, is that the asymptotic variance of $\hat{\varphi}$ has to be twice the asymptotic variance of $\varphi$. We do not see any reason for this to be true in full generality. In the particular case where all maps $T_{\omega}$ preserve the measure $\mu$, this can be proved using Green-Kubo formula from Proposition \ref{green_kubo} : assuming that the auxiliary system is mixing and has a spectral gap on an appropriated Banach space, the stationary measure is then given by $\mu \otimes \mu$ (since it is preserved by all maps $\hat{T}_{\omega}$), and an algebraic manipulation using Proposition \ref{green_kubo} shows that the asymptotic variance of $\hat{\varphi}$ is given by $2 \sigma^2$.  See \cite{ALS} for a similar computation.

In the general situation, the stationary measure of the auxiliary measure can be different from $\mu \otimes \mu$, and it seems hard to compute the asymptotic variance of $\hat{\varphi}$ from Green-Kubo formula. Even though this condition can seem unnatural, it is also necessary in order for the quenched central limit theorem to be true in the form we have stated it, as can be seen from the proof of \cite{ALS}. We state this as a lemma :

\begin{lemma}

Using the same notations introduced above, assume that there exists  $\sigma^2 >0$ and $\hat{\sigma}^2 > 0$ such that 

\begin{enumerate}

\item $\frac{S_n}{\sqrt{n}}$ converges in law to $\mathcal{N}(0, \sigma^{2})$ under the probability $\tilde{\mathbb{P}} \otimes \mu$,

\item $\frac{\hat{S}_n}{\sqrt{n}}$ converges in law to  $\mathcal{N}(0, \hat{\sigma}^{2})$ under the probability $\tilde{\mathbb{P}} \otimes (\mu \otimes \mu)$,

\item for a.e. $\underline{\omega}$, $\frac{1}{\sqrt{n}} \sum_{k=0}^{n-1} \varphi \circ T_{\underline{\omega}}^k$ converges in law to $\mathcal{N}(0, \sigma^{2})$ under the probability $\mu$.

\end{enumerate}

Then $\hat{\sigma}^2 = 2 \sigma^2$.

\end{lemma}

\begin{proof}

Define $S_{n,{\underline{\omega}}} = S_n(\underline{\omega}, .) = \sum_{k=0}^{n-1} \varphi \circ T_{\underline{\omega}}^k$. Following \cite{ALS}, we write for any $t \in \mathbb{R}$ and $n \ge 1$ :

\[
\begin{aligned}
\mathbb{E}_{\tilde{\mathbb{P}}} \left(\left|\mu(e^{i \frac{t}{\sqrt{n}} S_{n,\underline{\omega}}}) - e^{- \frac{\sigma^{2} t^2}{2}}\right|^2 \right) = \mathbb{E}_{\tilde{\mathbb{P}}}\left( \left|\mu(e^{i \frac{t}{\sqrt{n}} S_n}) \right|^2\right) - e^{-t^2 \sigma^2} + 2 e^{- \frac{\sigma^2 t^2}{2}} \Re \left(e^{- \frac{\sigma^2 t^2}{2}} - \mathbb{E}_{\tilde{\mathbb{P}} \otimes \mu}(e^{i \frac{t}{\sqrt{n}} S_n})\right)
\\
= \mathbb{E}_{\tilde{\mathbb{P}}} \left( \mu \otimes \mu( e^{i \frac{t}{\sqrt{n}} \hat{S}_n})\right) - e^{- \frac{t^2 \hat{\sigma}^2}{2}} + \left(  e^{- \frac{t^2 \hat{\sigma}^2}{2}} - e^{- t^2 \sigma^2} \right) + 2 e^{- \frac{\sigma^2 t^2}{2}} \Re \left(e^{- \frac{\sigma^2 t^2}{2}} - \mathbb{E}_{\tilde{\mathbb{P}} \otimes \mu}(e^{i \frac{t}{\sqrt{n}} S_n})\right).
\end{aligned}
\]
By the two first assumptions, this term goes to $ e^{- \frac{t^2 \hat{\sigma}^2}{2}} - e^{-t^2 \sigma^2}$ as $n$ goes to infinity. But $\mathbb{E}_{\tilde{\mathbb{P}}} \left(\left|\mu(e^{i \frac{t}{\sqrt{n}}  S_{n,\underline{\omega}}}) - e^{- \frac{\sigma^{2} t^2}{2}}\right|^2 \right)$ goes to $0$ thanks to the third assumption and the dominated convergence theorem. This shows that $e^{- \frac{t^2 \hat{\sigma}^2}{2}} = e^{-t^2 \sigma^2}$. \qed
\end{proof}

In the following, we will consider the situation where $X = [0,1]$ and all maps preserve the Lebesgue measure $m$. For technical convenience, we will also assume that $\Omega$ is a finite set.

\begin{example} Suppose that all maps $T_{\omega}$ are given by $T_{\omega} x = \beta_{\omega} x ~ {\rm mod ~ 1}$, where $\beta_{\omega} > 1$ is an integer. The transfer operator of this system clearly satisfies a Lasota-Yorke on the space ${\rm BV}$, and is random-covering, so that assumption (1) and (3) follows automatically for any $\varphi \in {\rm BV}$. On the other hand, the auxiliary two-dimensional system has a spectral gap on the quasi-H\"{o}lder space $V_1(X^2)$ and is also random covering. Since $\hat{\varphi}$ belongs to $V_1(X^2)$, assumption (2) follows by the above discussion and the quenched CLT holds.
\end{example}

\begin{example}
There exist piecewise non-linear expanding maps which preserves Lebesgue. Such a class of examples is provided by the Lorenz-like maps considered in the paper \cite{CHMV} : these maps have both a neutral parabolic fixed point and a point where the derivative goes to infinity. The coexistence of these two behaviors allows the possibility for the map to preserve Lebesgue measure while being non-linear.
Suppose that $\Omega = \{0,1\}$, $T_0$ is the doubling map $T_0 x = 2x ~ {\rm mod} ~ 1$, and that $T_1$ is one of the maps considered in \cite{CHMV}.
We will prove that there exists $0 \le p^{\star} < 1$ such that if $T_0$ is iterated with probability $p$ with $p > p_{\star}$, then  the quenched CLT holds for any observable $\varphi$ Lipschitz.

Since the annealed transfer operator $P$ can be written as $p P_0 + (1-p) P_1$, where $P_0$, resp. $P_1$, is the transfer operator of $T_0$, resp. $T_1$ (and similarly $\hat{P} = p \hat{P}_0 + (1-p) \hat{P}_1$ for the auxiliary system), it is sufficient to prove that $P_0$ and $\hat{P}_0$ have a spectral gap on Banach spaces $\mathcal{B}$ and $\hat{\mathcal{B}}$, while $P_1$ and $\hat{P}_1$ act continuously on these spaces, and that $\varphi \in \mathcal{B}$ and $\hat{\varphi} \in \hat{\mathcal{B}}$. We will use quasi-H\"{o}lder spaces and will take $\mathcal{B} = V_{\alpha}(X)$ and $\hat{\mathcal{B}} = V_{\alpha}(X^2)$ for a convenient choice of $\alpha$. Clearly, the transfer operator of $T_0$ and $\hat{T}_0$ have a spectral gap on these spaces, and $\varphi$ (resp. $\hat{\varphi}$) belongs to $\mathcal{B}$ (resp. $\hat{\mathcal{B}}$) whenever $\varphi$ is Lipschitz. To prove the continuity of $P_1$ and $\hat{P}_1$, we will use the following general result.

\begin{proposition} \label{cont_qh}
Let $M$ be a compact subset of $\mathbb{R}^d$ with $m_d(M) =1$, where $m_d$ denotes the Lebesgue measure on $\mathbb{R}^d$, and $T : M \to M$ be a non-singular map. Define $g(x) = \frac{1}{| {\rm det} DT(x)|}$, and assume there exist a finite family of disjoint open set $\{U_i\}_i$ included in $M$, a constant $C > 0$ and $0 < \alpha \le 1$ with

\begin{enumerate}

\item $m_d(\cup_i U_i) = 1$,

\item $T : U_i \to TU_i$ is a $C^1$-diffeomorphism,

\item $d(Tx,Ty) \ge d(x,y)$ for all $i$ and all $x,y \in U_i$,

\item $|g(x) - g(y) | \le C d(x,y)^{\alpha}$, for all $i$ and all $x,y \in U_i$,

\item $m_d(B_{\epsilon}(\partial TU_i) ) \le C \epsilon^{\alpha}$ for all $i$ and all $\epsilon > 0$.

\end{enumerate}

Then the transfer operator of $T$ acts continuously on $V_{\alpha}(M)$.

\end{proposition}

The map with parameter $\gamma > 1$ considered in \cite{CHMV} satisfies these assumptions for $\alpha = \min\{1, \gamma - 1\}$, so that the quenched CLT holds when $p^{\star}$ is close enough to $1$.

\begin{proof}[Proof of Proposition \ref{cont_qh}]

We denote by $T_i^{-1} : TU_i \to U_i$ the inverse branch of $T$ restricted to $U_i$.

The transfer operator $P$ of $T$ reads as $$Pf(x) = \sum_i (gf) \circ T_i^{-1} \mathds{1}_{TU_i}(x).$$ Following Saussol \cite{S}, we have for all $\epsilon > 0$ and $x \in \mathbb{R}^d$ : $$\mbox{osc}(Pf, B_{\epsilon}(x)) \le \sum_i R_i^{(1)}(x) \mathds{1}_{TU_i}(x) + 2 \sum_i R_i^{(2)}(x),$$ where $R_i^{(1)}(x) = \mbox{osc}(gf, T^{-1} B_{\epsilon}(x) \cap U_i)$ and $R_i^{(2)}(x) = \left(\underset{T^{-1} B_{\epsilon}(x) \cap U_i} {\rm ess ~ sup} |gf|\right) \mathds{1}_{B_{\epsilon}(\partial TU_i)}(x)$.

Using Proposition 3.2 (iii) in \cite{S}, we get $$R_i^{(1)}(x) \le \mbox{osc}(f, T^{-1} B_{\epsilon}(x) \cap U_i) \underset{T^{-1} B_{\epsilon}(x) \cap U_i}{\rm ess ~ sup} g + \mbox{osc}(g, T^{-1} B_{\epsilon}(x) \cap U_i) \underset{T^{-1} B_{\epsilon}(x) \cap U_i} {\rm ess ~ inf} |f|.$$

By assumption (3), we have $T^{-1}B_{\epsilon}(x) \cap U_i \subset B_{\epsilon}(T_i^{-1} x)$, while by assumption (4), $\mbox{osc}(g, T^{-1}B_{\epsilon}(x) \cap U_i) \le C \epsilon^{\alpha}$ and $\underset{T^{-1} B_{\epsilon}(x) \cap U_i}{\rm ess ~ sup} g \le g(T_i^{-1} x) + C \epsilon^{\alpha}$.

This shows $R_i^{(1)}(x) \le g(T_i^{-1} x) \mbox{osc}(f, B_{\epsilon}(T_i ^{-1} x) + C \epsilon^{\alpha} \|f \|_{\rm sup}$, whence $$\int \sum_i R_i^{(1)}(x) \mathds{1}_{TU_i}(x) dx \le \int P( \mbox{osc}(f, B_{\epsilon}(.))(x) dx + C \|f \|_{\infty} \epsilon^{\alpha} \sum_i m_d(TU_i).$$ Since the sum is finite, this gives $\int \sum_i R_i^{(1)}(x) \mathds{1}_{TU_i}(x) dx \le \epsilon^{\alpha} \left(|f|_{\alpha} + C \|f\|_{\rm sup}\right) \le C \epsilon^{\alpha} \|f \|_{\alpha}$.

We turn now to the estimate of $R_i^{(2)}$ : one has $R_i^{(2)}(x) \le \|g \|_{\rm sup} \|f \|_{\rm sup} \mathds{1}_{B_{\epsilon }(\partial TU_i)}(x)$, so that using assumption (4), $\int \sum_i R_i^{(2)} dx \le C \| f\|_{\rm sup} \sum_i m_d(B_{\epsilon}(\partial TU_i)) \le C \epsilon^{\alpha} \|f \|_{\rm sup}$. This shows that $|Pf|_{\alpha} \le C \|f\|_{\alpha}$ and concludes the proof. \qed

\end{proof}

\end{example}

\section{Concentration inequalities} \label{concentration}

A function $K : X^n \to \mathbb{R}$, where $(X,d)$ is a metric space, is separately Lipschitz if, for all $i$, there exists a constant $\lip_i(K)$ with $$ \left| K(x_0, \ldots, x_{i-1}, x_i, x_{i+1}, \ldots, x_{n-1}) - K(x_0, \ldots, x_{i-1}, x_i', x_{i+1}, \ldots, x_{n-1}) \right| \le \lip_i(K) d(x_i, x_i')$$ for all points $x_0, \ldots, x_{n-1}, x_i'$ in $X$.

Let $(\Omega, \mathbb{P}, T)$ be a finite random Lasota-Yorke system on the unit interval $X = [0,1]$, such that $\lambda(T_{\omega}) > 1$ for all $\omega \in \Omega$. We assume that $(\Omega, \mathbb{P}, T)$ satisfies the random covering property, and we denote by $\mu$ its unique absolutely continuous stationary measure. Its density $h$ belongs to $BV$, and is uniformly bounded away from $0$.

\begin{theorem}  \label{exp_CI}
There exists a constant $C \ge 0$, depending only on $(\Omega, \mathbb{P}, T)$, such that for any $n \ge 1$ and any separately Lipschitz function $K : X^n \to \mathbb{R}$, one has $$\mathbb{E}_{\mu \otimes \tilde{\mathbb{P}}}\left( e^{K(x, T_{\underline{\omega}}^1 x, \ldots, T_{\underline{\omega}}^{n-1} x) - \mathbb{E}_{\mu \otimes \tilde{\mathbb{P}}}\left( K(x, T_{\underline{\omega}}^1 x, \ldots, T_{\underline{\omega}}^{n-1} x)\right)}\right) \le e^{C \sum_{i=0}^{n-1} \lip_i^2(K)}$$
\end{theorem}

This leads to a large deviation estimate, namely that for all $t > 0$, one has $$\mu \otimes \tilde{\mathbb{P}} \left( \{ (x, \underline{\omega}) \, / \, K(x, T_{\underline{\omega}}^1 x, \ldots, T_{\underline{\omega}}^{n-1} x) - m > t \} \right) \le e^{- \frac{t^2}{4C \sum_{i=0}^{n-1} \lip_i^2(K)}},$$ where $m = \mathbb{E}_{\mu \otimes \tilde{\mathbb{P}}}\left( K(x, T_{\underline{\omega}}^1 x, \ldots, T_{\underline{\omega}}^{n-1} x)\right)$.

For the proof, we will use McDiarmid's bounded differences method \cite{McDi1, McDi2}, as in \cite{CG} and \cite{CMS}, conveniently adapted to the random context.

We will denote by $P$ the annealed transfer operator with respect to the Lebesgue measure $m$, and by $L$ the annealed tranfer operator with respect to the stationary measure $\mu$. Recall that $L$ acts on functions in $L^1(\mu)$ by $L(f) = \frac{P(fh)}{h}$, whence $$Lf(x) = \sum_{\omega \in \Omega} p_{\omega} \sum_{T_{\omega}y = x} \frac{h(y)f(y)}{h(x) |T_{\omega}'(y)|}.$$ Since $h$ belongs to $BV$, together with $\frac{1}{h}$, $L$ acts on BV and has a spectral gap.

Recall the construction of the symbolic system $(X^{\mathbb{N}}, \mathcal{F}, \sigma, \mu_c)$ and of the decreasing filtration $\{\mathcal{F}_p\}_{p \ge 0}$ of $\sigma$-algebras. We extend $K$ as a function on $X^{\mathbb{N}}$, depending only on the $n$ first coordinates. One has obviously $\mathbb{E}_{\mu_c}(K) = \mathbb{E}_{\mu \otimes \tilde{\mathbb{P}}}\left( K(x, T_{\underline{\omega}}^1 x, \ldots, T_{\underline{\omega}}^{n-1} x)\right)$ and $$\mathbb{E}_{\mu \otimes \tilde{\mathbb{P}}}\left( e^{K(x, T_{\underline{\omega}}^1 x, \ldots, T_{\underline{\omega}}^{n-1} x) - \mathbb{E}_{\mu \otimes \tilde{\mathbb{P}}}\left( K(x, T_{\underline{\omega}}^1 x, \ldots, T_{\underline{\omega}}^{n-1} x)\right)}\right) = \mathbb{E}_{\mu_c}(e^{K - \mathbb{E}_{\mu_c}(K)}),$$ since $\Phi : X \times \tilde{\Omega} \to X^{\mathbb{N}}$ is a factor map. We define $K_p = \mathbb{E}_{\mu_c}(K | \mathcal{F}_p)$, and $D_p =  K_p - K_{p+1}$. One has the following :

\begin{lemma}
The dynamical system $(X^{\mathbb{N}}, \mathcal{F}, \sigma, \mu_c)$ is exact.
\end{lemma}

\begin{proof} This follows from exactness of the skew-product system $(X \times \tilde{\Omega}, S, \mu \otimes \tilde{\mathbb{P}})$, see theorem 5.1 in \cite{Mor1}, and the fact that $\Phi : X \times \tilde{\Omega} \to X^{\mathbb{N}}$ is a factor map. See also theorem 4.1 in \cite{Mor2}. \qed
\end{proof}

This implies that $\mathcal{F}_{\infty} := \bigcap_{p\ge 0} \mathcal{F}_p = \bigcap_{p \ge 0} \sigma^{-p} \mathcal{F}$ is $\mu_c$-trivial, from which we deduce, by Doob's convergence theorem, that $K_p$ goes to $\mathbb{E}_{\mu_c}(K)$ $\mu_c$-as when $p$ goes to infinity, whence $K - \mathbb{E}_{\mu_c}(K) = \sum_{p \ge 0} D_p$.

From Azuma-Hoeffding's inequality (see lemma 4.1 in \cite{McDi1} and its proof for the bound of the exponential moment), we deduce that there exists some $C \ge 0$  such that for all $P \ge 0$,  $$\mathbb{E}_{\mu_c}(e^{\sum_{p=0}^P D_p}) \le e^{C \sum_{p=0}^P \sup |D_p|^2}.$$
It remains to bound $D_p$ :

\begin{proposition}  There exists $\rho < 1$ and $C \ge 0$, depending only on $(\Omega, \mathbb{P}, T)$, such that for all $p$, one has $$|D_p| \le C \sum_{j=0}^p \rho^{p-j} \lip_i(K).$$
\end{proposition}

This proposition, together with the Cauchy-Schwarz inequality, implies immediately the desired concentration inequality, in the same manner as in \cite{CG}. The following lemma leads immediately to the result, using the Lipschitz condition on $K$ :

\begin{lemma}  There exists $\rho < 1$ and $C \ge 0$, depending only on $(\Omega, \mathbb{P}, T)$, such that for all $p$ and $x_p, ...$ , one has $$\left| K_p(x_p, \ldots) - \int_{\tilde{\Omega}} \int_X K(y, T_{\underline{\omega}}^1 y, \ldots, T_{\underline{\omega}}^{p-1} y, x_p, \ldots ) d\mu(y) d\tilde{\mathbb{P}}(\underline{\omega}) \right| \le C \sum_{j=0}^{p-1} \lip_j(K) \rho^{p-j}.$$

\end{lemma}

The rest of this section is devoted to the proof of this lemma.
For a sequence $\underline{\omega} \in \tilde{\Omega}$, we denote $g^{(p)}_{\underline{\omega}}(y) = \frac{h(y)}{h(T_{\underline{\omega}}^p y)} \frac{1}{|(T_{\underline{\omega}}^p)'(y)|}$. We have $$K_p(x_p, \ldots) = \sum_{\underline{\omega} \in \Omega^p} p_{\underline{\omega}}^p \sum_{T_{\underline{\omega}}^p y =x} g_{\underline{\omega}}^{(p)}(y) K(y, T_{\underline{\omega}}^1 y, \ldots, T_{\underline{\omega}}^{p-1} y, x_p \ldots).$$

We fix a $x_{\star} \in X$, and we decompose $K_p$ as $$K_p(x_p, \ldots) = K(x_{\star}, \ldots, x_{\star}, x_p, \ldots) + \sum_{i=0}^{p-1} \sum_{\underline{\omega} \in \Omega^p} p_{\underline{\omega}}^p \sum_{T_{\underline{\omega}}^p y =x} g_{\underline{\omega}}^{(p)}(y) H_i(y, \ldots, T_{\underline{\omega}}^i y),$$ where $H_i(y_0, \ldots, y_i) = K(y_0, \ldots, y_i, x_{\star}, \ldots, x_{\star}, x_p, \ldots) - K (y_0, \ldots, y_{i-1}, x_{\star}, \ldots, x_{\star}, x_p, \ldots)$.

A simple computation then shows that $K_p(x_p, \ldots) = K(x_{\star}, \ldots, x_{\star}, x_p, \ldots) + \sum_{i=0}^{p-1} L^{p-i} f_i(x_p)$, with $$f_i(y) = \sum_{\underline{\omega} \in \Omega^i} p_{\underline{\omega}}^i \sum_{T_{\underline{\omega}}^i z = y } g_{\underline{\omega}}^{(i)}(z) H_i(z, \ldots, T_{\underline{\omega}}^i z).$$

From the spectral gap of $L$, we deduce that there exists $C \ge 0$ and $\rho < 1$ depending only on the system, such that $ \| L^{p-i} f_i - \int_X f_i d \mu \|_{{\rm BV}} \le C \rho^{p-i} \|f_i\|_{{\rm BV}}$. On one hand, since the $\rm BV$-norm dominates the supremum norm, one has $\left|  L^{p-i} f_i(x_p) - \int_X f_i d \mu\right| \le C \rho^{p-i} \|f_i\|_{{\rm BV}}$. On the other hand, one has easily $\int_X f_i d\mu = \int_{\tilde{\Omega}} \int_X H_i(y, \ldots, T_{\underline{\omega}}^i y) d\mu(y) d\tilde{\mathbb{P}}(\underline{\omega})$, from which it follows, summing all the relations, that $$\left| K_p(x_p, \ldots) - \int_{\tilde{\Omega}} \int_X K(y, T_{\underline{\omega}}^1 y, \ldots, T_{\underline{\omega}}^{n-1} y, x_p, \ldots ) d\mu(y) d\tilde{\mathbb{P}}(\underline{\omega}) \right| \le C \sum_{i=0}^{p-1} \rho^{p-i} \|f_i\|_{\rm BV}.$$

It remains to estimate $\|f_i\|_{\rm BV} \le \|f_i\|_{\rm sup} + {\rm Var}(f_i)$. For this, we'll need a technical lemma. For $\omega \in \Omega$, we denote by $\mathcal{A}_{\omega}$ the partition of monotonicity of $T_{\omega}$, and for $\underline{\omega} \in \Omega^{\mathbb{N}}$, we define $\mathcal{A}_{\underline{\omega}}^{n-1} = \bigvee_{k = 0}^{n-1} \left(T_{\underline{\omega}}^k \right)^{-1} \left(\mathcal{A}_{\omega_{k+1}}\right)$, which is the partition of monotonicity of $T_{\underline{\omega}}^n$.

If $T$ is a Lasota-Yorke map of the interval, with partition of monotonicity $\mathcal{A}$, we define its distorsion ${\rm Dist}(T)$ as the least constant $C$ such that $|T'(x) - T'(y)| \le C |T'(x)| |Tx - Ty|$ for all $x,y \in I$ and $I \in \mathcal{A}$.

\begin{lemma}  \label{conc_lem} There exists $\lambda > 1$ and $C \ge 0$ so that, for all $\underline{\omega} \in \Omega^{\mathbb{N}}$ and $n \ge 0$ :

\begin{enumerate}

\item $\lambda(T_{\underline{\omega}}^n) \ge \lambda^n$,

\item ${\rm Dist}(T_{\underline{\omega}}^n) \le C$,

\item $\sum_{\underline{\omega} \in \Omega^n} p_{\underline{\omega}}^n \sum_{I \in \mathcal{A}_{\underline{\omega}}^{n-1}} \sup_I \frac{1}{|(T_{\underline{\omega}}^n)'|} \le C$,

\item $\sum_{\underline{\omega} \in \Omega^n} p_{\underline{\omega}}^n\sum_{I \in \mathcal{A}_{\underline{\omega}}^{n-1}} {\rm Var}_I \left( \frac{1}{|(T_{\underline{\omega}}^n)'|} \right) \le C$.

\end{enumerate}

\begin{proof} \em

\begin{enumerate}

\item is obvious, since $\Omega$ is a finite set.

\item This is a classical computation. It follows from $(1)$ and the chain rule.

\item This an easy adaptation of lemma II.4 in \cite{CMS}. For any $\underline{\omega} \in \Omega^n$, and $I \in \mathcal{A}_{\underline{\omega}}^{n-1}$, there exists a least integer $p = p_{\underline{\omega}, I}$ such that $T_{\underline{\omega}}^{p}(I) \cap \partial \mathcal{A}_{\omega_{p+1}} \neq \emptyset$. We denote by $\mathcal{A}_{\underline{\omega}}^{n-1,p}$ the set of all $I \in \mathcal{A}_{\underline{\omega}}^{n-1}$ for which we have $p = p_{\underline{\omega}, I}$. We define $\partial = \cup_{\omega \in \Omega} \partial \mathcal{A}_{\omega}$. Fix $I \in \mathcal{A}_{\underline{\omega}}^{n-1,p}$. There exists $a \in \partial I$ such that $b = T_{\underline{\omega}}^p a \in \partial$. From (2), we deduce the existence of a constant $C$, depending only on the system, such that, for any $x \in I$, $$|(T_{\underline{\omega}}^n)' (x)| \ge C |(T_{\underline{\omega}}^n)' (a)| = C |(T_{\omega_n} \circ \ldots \circ T_{\omega_{p+1}})' (b)| |(T_{\underline{\omega}}^p)' (a)| \ge C \lambda^{n-p} |(T_{\underline{\omega}}^p)' (a)|.$$ One has then $\sup_I \frac{1}{|(T_{\underline{\omega}}^n)'|} \le C^{-1} \lambda^{-(n-p)} \frac{1}{|(T_{\underline{\omega}}^p)'(a)|}$. Since a pre-image by $T_{\underline{\omega}}^p$ of an element $b \in \partial$ can only belong to at most two different $I \in \mathcal{A}_{\underline{\omega}}^{n-1}$, it follows $$\begin{aligned} &\sum_{\underline{\omega} \in \Omega^n} p_{\underline{\omega}}^n \sum_{I \in \mathcal{A}_{\underline{\omega}}^{n-1}} \sup_I \frac{1}{|(T_{\underline{\omega}}^n)'|}& &\le& &2 C^{-1} \sum_{p=0}^{n-1} \lambda^{-(n-p)} \sum_{b \in \partial} \sum_{\underline{\omega} \in \Omega^n} p_{\underline{\omega}}^n  \sum_{T_{\underline{\omega}}^p a = b} \frac{1}{|(T_{\underline{\omega}}^p)'(a)|}& \\ &&  &=& &2 C^{-1} \sum_{p=0}^{n-1} \lambda^{-(n-p)} \sum_{b \in \partial} P^p \mathds{1} (b).& \end{aligned}$$ This quantity is bounded, since $P$ is power bounded, and $\partial$ is a finite set.

\item It follows from the three previous points, and the definition of the total variation.
\end{enumerate}\qed
\end{proof}

\end{lemma}

Since $L^i \mathds{1} = \mathds{1}$, one has $\| f_i \|_{\rm sup} \le \| H_i \|_{\rm sup} \le \lip_i(K)$. The crucial point lies in the estimate of the variation of $f_i$. We first note that $$ {\rm Var} (f_i) \le {\rm Var}(\frac{1}{h}) \|h f_i \|_{\rm sup} + \| \frac{1}{h} \|_{\rm sup} {\rm Var}( h f_i).$$ Since $\|h f_i\|_{\rm sup} \le \lip_i(K) \| P^i h \|_{\rm sup} \le C \lip_i(K)$, one has just to estimate ${\rm Var}(h f_i)$.

For $\underline{\omega} \in \Omega^i$ and $I \in \mathcal{A}_{\underline{\omega}}^{i-1}$, we denote by $S_{i, I, \underline{\omega}}$ the inverse branch of $T_{\underline{\omega}}^i$ restricted to $I$. We define also $H_{i, \underline{\omega}}(z) = H_i(z, \ldots, T_{\underline{\omega}}^i z)$.

Then, we can write $$h f_i = \sum_{\underline{\omega} \in \Omega^i} p_{\underline{\omega}}^i \sum_{I \in \mathcal{A}_{\underline{\omega}}^{i-1}} \left( \frac{h H_{i, \underline{\omega}}}{|(T_{\underline{\omega}}^i)'|} \right) \circ S_{i, I, \underline{\omega}} \,  \mathds{1}_{T_{\underline{\omega}}^i(I)}.$$

It follows that $$\begin{aligned} {\rm Var}(h f_i) &\le& &\sum_{\underline{\omega} \in \Omega^i} p_{\underline{\omega}}^i \left( \sum_{I \in \mathcal{A}_{\underline{\omega}}^{i-1}} {\rm Var}_I \left( \frac{h H_{i, \underline{\omega}}}{|(T_{\underline{\omega}}^i)'|} \right) + 2 \sum_{a \in \partial \mathcal{A}_{\underline{\omega}}^{i-1}} \frac{|h(a)| |H_{i,\underline{\omega}}(a)|}{|(T_{\underline{\omega}}^i)'(a)|} \right)& \\ & \le & &\sum_{\underline{\omega} \in \Omega^i} p_{\underline{\omega}}^i \left( {\rm I}_{\underline{\omega},i} + {\rm II}_{\underline{\omega},i} + {\rm III}_{\underline{\omega},i} + {\rm IV}_{\underline{\omega},i}\right)&, \end{aligned}$$ where

$$\begin{aligned} &{\rm I}_{\underline{\omega},i}& &=& & \sum_{I \in \mathcal{A}_{\underline{\omega}}^{i-1}} {\rm Var}_I (h) \, \sup_I \frac{1}{|(T_{\underline{\omega}}^i)'|} \, \sup_I |H_{i,\underline{\omega}}|,&  \\
&{\rm II}_{\underline{\omega},i}& &=&  &\sum_{I \in \mathcal{A}_{\underline{\omega}}^{i-1}} \sup_I h \, {\rm Var}_I\left(\frac{1}{|(T_{\underline{\omega}}^i)'|}\right) \, \sup_I |H_{i,\underline{\omega}}|,&  \\
&{\rm III}_{\underline{\omega},i}& &=&  &\sum_{I \in \mathcal{A}_{\underline{\omega}}^{i-1}} \sup_I h \, \sup_I \frac{1}{|(T_{\underline{\omega}}^i)'|} \, {\rm Var}_I(H_{i,\underline{\omega}}), & \\
&{\rm IV}_{\underline{\omega},i}& &=& &2 \sum_{a \in \partial \mathcal{A}_{\underline{\omega}}^{i-1}} \frac{|h(a)| |H_{i,\underline{\omega}}(a)|}{|(T_{\underline{\omega}}^i)'(a)|}.& \end{aligned}$$

Using the Lipschitz condition for $K$, one gets ${\rm I}_{\underline{\omega},i} \le C \lip_i(K) \sum_{I \in \mathcal{A}_{\underline{\omega}}^{i-1}} \sup_I \frac{1}{|(T_{\underline{\omega}}^i)'|}$, which gives, by lemma \ref{conc_lem}, $\sum_{\underline{\omega} \in \Omega^i} p_{\underline{\omega}}^i \, {\rm I}_{\underline{\omega},i} \le C \lip_i(K)$. The same argument applies to prove that $\sum_{\underline{\omega} \in \Omega^i} p_{\underline{\omega}}^i  \, {\rm II}_{\underline{\omega},i} \le C \lip_i(K)$.

We turn now to the estimate of ${\rm III}_{\underline{\omega}, i}$. Let $y_0 < \ldots < y_l$ be a sequence of points of $I$. In order to estimate $\sum_{j=0}^{l-1} | H_{i,\underline{\omega}}(y_{j+1}) -H_{i,\underline{\omega}}(y_j) |$, we split $H_{i, \underline{\omega}}$ into two terms in an obvious way, and we deal with the first one, the second being completely similar. We have $$\begin{aligned} &\sum_{j=0}^{l-1} \sum_{k=0}^i |K(y_{j+1}, \ldots, T_{\underline{\omega}}^k y_{j-1}, T_{\underline{\omega}}^{k+1} y_j, \ldots,T_{\underline{\omega}}^i y_j, \ldots ) - K(y_{j+1}, \ldots, T_{\underline{\omega}}^{k-1} y_{j+1}, T_{\underline{\omega}}^k y_j, \ldots,T_{\underline{\omega}}^i y_j, \ldots ) |& \\ &\le \sum_{j=0}^{l-1} \sum_{k=0}^i \lip_k(K) \left|T_{\underline{\omega}}^k y_{j+1} - T_{\underline{\omega}}^k y_j \right| = \sum_{k=0}^i \lip_k(K) m(T_{\underline{\omega}}^k(I)).&  \end{aligned} $$

Since $I \in \mathcal{A}_{\underline{\omega}}^{i-1}$, $T_{\underline{\omega}}^k(I)$ is included in an interval of monotonicity of $T_{\omega_i} \circ \ldots \circ T_{\omega_{k+1}}$, and hence its length is less than $\left( \lambda(T_{\omega_i} \circ \ldots \circ T_{\omega_{k+1}}) \right)^{-1} \le \lambda^{-(i-k)}$.
Therefore, one has ${\rm Var}_I(H_{i, \underline{\omega}}) \le \sum_{k=0}^i \lambda^{-(i-k)} \lip_k(K)$. An application of lemma \ref{conc_lem} shows that $\sum_{\underline{\omega} \in \Omega^i} p_{\underline{\omega}}^i  \, {\rm III}_{\underline{\omega},i} \le C \sum_{k=0}^i \lambda^{-(i-k)} \lip_k(K)$.

Using again lemma \ref{conc_lem} and Lipschitz condition on $K$, we can bound the last term by $\sum_{\underline{\omega} \in \Omega^i} p_{\underline{\omega}}^i  \, {\rm IV}_{\underline{\omega},i} \le C \lip_i(K)$.

Finally, putting together all the estimates, we find that ${\rm Var}(h f_i) \le C \sum_{k=0}^i \lambda^{-(i-k)} \lip_k(K)$, which gives ${\rm Var}(f_i) \le  C \sum_{k=0}^i \lambda^{-(i-k)} \lip_k(K)$, and the same estimate for $\|f_i\|_{\rm BV}$.

We then have $$\left| K_p(x_p, \ldots) - \int_{\tilde{\Omega}} \int_X K(y, T_{\underline{\omega}}^1 y, \ldots, T_{\underline{\omega}}^{p-1} y, x_p, \ldots ) d\mu(y) d\tilde{\mathbb{P}}(\underline{\omega}) \right| \le C \sum_{i=0}^{p-1}  \rho^{p-i} \sum_{k=0}^i \lambda^{-(i-k)} \lip_k(K).$$ A simple calculation shows that this term is less than $C \sum_{k=0}^{p-1} (\rho')^{p-k} \lip_k(K)$, for $\max(\rho, \lambda^{-1}) < \rho ' < 1$. This concludes the proof. \qed

Concentration inequalities have several statistical applications concerning the empirical measure, the shadowing, the integrated periodogram, the correlation dimension, the kernel density estimation, the almost-sure CLT, ...
We describe here an application to the rate of convergence of the empirical measure to the stationary measure, and refer the reader to \cite{CC, CCS, CG, CMS} for others possibilities. We also mention the work of Maldonado \cite{Mal}, where concentration inequalities are proved in a random context. He considers the so-called observational noise, where the randomness doesn't affect the dynamics, but only the observations, so the setup is somewhat different from ours, but once an annealed concentration inequality is established, all consequences are derived in a similar way.

The empirical measure is the random measure defined by $$\mathcal{E}_n(x,  \underline{\omega}) = \frac{1}{n} \sum_{j=0}^{n-1} \delta_{T_{\underline{\omega}}^j x}.$$

Since the skew-product system $(X \times \tilde{\Omega}, S, \mu \otimes \tilde{\mathbb{P}})$ is ergodic, it follows from Birkhoff's theorem that $\mathcal{E}_n(x,  \underline{\omega})$ converges weakly to the stationary measure $\mu$, for $\mu \otimes \tilde{\mathbb{P}}$-ae $(x,  \underline{\omega})$. For statistical purposes, it proves useful to estimate the speed of  this convergence. We introduce the Kantorovitch distance $\kappa$ on the space of probability measures on $[0,1]$. For any $\nu_1, \nu_2$ probabilities measure on the unit interval, their Kantorovitch distance $\kappa(\nu_1, \nu_2)$ is equal to $$\kappa(\nu_1, \nu_2) = \int_0^1 |F_{\nu_1}(t) - F_{\nu_2}(t)| dt,$$ where $F_{\nu}(t) = \nu([0,t])$ is the distribution function of $\nu$. We show the following :

\begin{proposition}

The exists $t_0 > 0$ and $C > 0$ such that for all $t > t_0$ and $n \ge 1$ :$$\mu \otimes \tilde{\mathbb{P}} \left( \{ (x,  \underline{\omega}) \, / \, \kappa(\mathcal{E}_n(x,  \underline{\omega}), \mu) > \frac{t}{\sqrt{n}} \} \right) \le e^{-C t^2}.$$

\end{proposition}

\begin{proof}

We follow closely the proof of Theorem III.1 in \cite{CMS}.
For $t \in [0,1]$, define the function of $n$ variables $$K_n(x_0, \ldots, x_{n-1}) = \int_0^1 | F_{n,t}(x_0, \ldots, x_{n-1}) - F_{\mu}(t)| dt,$$ where $F_{n,t}$ is given by $$F_{n,t}(x_0, \ldots, x_{n-1}) = \frac{1}{n} \sum_{k=0}^{n-1} \mathds{1}_{[0,t]} (x_k).$$

We clearly have $\kappa(\mathcal{E}_n(x,  \underline{\omega}), \mu) = K_n(x, \ldots, T_{\underline{\omega}}^{n-1} x)$, and $\lip_j(K_n) \le \frac{1}{n}$ for any $0 \le j \le n-1$. We derive immediately from the exponential concentration inequality (see the remark just below Theorem \ref{exp_CI}) that $$\mu \otimes \tilde{\mathbb{P}} \left( \{ (x,  \underline{\omega}) \, / \, \kappa(\mathcal{E}_n(x,  \underline{\omega}), \mu) - \mathbb{E}_{\mu \otimes \tilde{\mathbb{P}}} \left(\kappa(\mathcal{E}_n(.), \mu) \right)  >  \frac{t}{\sqrt{n}} \} \right) \le e^{-C t^2}.$$

To conclude, it is then sufficient to prove that $\mathbb{E}_{\mu \otimes \tilde{\mathbb{P}}} \left(\kappa(\mathcal{E}_n(.), \mu) \right)$ is of order $\frac{1}{\sqrt{n}}$.

Using Schwartz inequality, we have $$\begin{aligned} \mathbb{E}_{\mu \otimes \tilde{\mathbb{P}}} \left(\kappa(\mathcal{E}_n(.), \mu) \right) &=& &\int_0^1 \left( \int_{X \times \tilde{\Omega}} |F_{n,t}(x, \ldots, T_{\underline{\omega}}^{n-1} x) - F_{\mu}(t)| d\mu(x) d\tilde{\mathbb{P}}(\underline{\omega}) \right) dt& \\ & \le & & \left[ \int_0^1 \left( \int_{X \times \tilde{\Omega}} |F_{n,t}(x, \ldots, T_{\underline{\omega}}^{n-1} x) - F_{\mu}(t)|^2 d\mu(x) d\tilde{\mathbb{P}}(\underline{\omega}) \right) dt \right]^{\frac{1}{2}}. \end{aligned}$$

Expanding the square and using the invariance of $\mu \otimes \tilde{\mathbb{P}}$ by the skew-product, we obtain $$\begin{aligned} \int_{X \times \tilde{\Omega}} |F_{n,t}(x, \ldots, T_{\underline{\omega}}^{n-1} x) - F_{\mu}(t)|^2 d\mu(x) d\tilde{\mathbb{P}}(\underline{\omega}) = \frac{1}{n} \int_0^1 (f_t - F_{\mu}(t))^2 d \mu \\ + \frac{2}{n} \sum_{k=1}^{n-1} \left( 1 - \frac{k}{n} \right) \int_X (f_t - F_{\mu}(t))( U^k f_t - F_{\mu}(t)) d \mu, \end{aligned}$$ where $f_t$ is the characteristic function of $[0,t]$ and $U^k f_t(x) = \int_{\tilde{\Omega}} f_t(T_{\underline{\omega}}^k x) d \tilde{\mathbb{P}}(\underline{\omega})$ as usual.

Since $F_{\mu}(t) = \int_X f_t d \mu$ and $f_t$ is bounded independently of $t$ in ${\rm BV}$, we can use exponential decay of annealed correlations to get $\int_X (f_t - F_{\mu}(t))( U^k f_t - F_{\mu}(t)) d \mu = \mathcal{O}(\lambda^k)$, where $\lambda < 1$, independently of $t$. This shows $ \int_{X \times \tilde{\Omega}} |F_{n,t}(x, \ldots, T_{\underline{\omega}}^{n-1} x) - F_{\mu}(t)|^2 d\mu(x) d\tilde{\mathbb{P}}(\underline{\omega}) = \mathcal{O}(n^{-1})$ and after integration over $t$, we finally get $\mathbb{E}_{\mu \otimes \tilde{\mathbb{P}}} \left(\kappa(\mathcal{E}_n(.), \mu) \right) = \mathcal{O}(n^{-\frac{1}{2}})$. \qed
\end{proof}


\begin{thebibliography}{99}

\bibitem{AD} J. Aaronson, M. Denker, {\em Local limit theorems for partial sums of stationary sequences generated by Gibbs-Markov maps}, Stoch. Dyn., {\bf 1}, 193-237, (2001)

\bibitem{ADSZ} J. Aaronson, M. Denker, O. Sarig, R. Zweimuller, {\em Aperiodicity of cocycles and conditional local limit theorems}, Stoch. Dyn., {\bf 4}, 31-62, (2004)

\bibitem{AV} R. Aimino,  S. Vaienti, {\em A note on the large deviations for piecewise expanding multidimensional maps}, Nonlinear Dynamics : New Directions, Theoretical Aspects, {\bf 1}, Edgardo Ugalde, Gelasio Salazar, Editors, Series  Mathematical Method and Modeling,  Springer,  11 p., http://arxiv.org/abs/1110.5488

\bibitem{AFLV} J.F. Alves, J.M. Freitas, S. Luzzato, S. Vaienti, {\em From rates of mixing to recurrence times via large deviations}, Adv. in Maths., {\bf 228}, 2, 1203-1236, (2011)

\bibitem{Ar} L. Arnold, {\em Random dynamical systems}, Berlin : Springer, (1998)

\bibitem{AS} A. Ayyer, M. Stenlund, {\em Exponential decay of correlations for randomly chosen hyperbolic toral automorphisms}, Chaos, {\bf 17}, (2007)

\bibitem{ALS} A. Ayyer, C. Liverani, M. Stenlund, {\em Quenched CLT for random toral automorphism}, Discrete and Continuous Dynamical Systems, {\bf 24}, 331-348, (2009)

\bibitem{BaGo} W. Bahsoun, P. G\"{o}ra, {\em Position dependent random maps in one and higher dimensions}, Studia Math., {\bf 166}, 271-286, (2005)

\bibitem{BY} V. Baladi, L.-S. Young, {\em On the spectra of randomly perturbed expanding maps}, Commun. Math. Phys., {\bf 156}, 355-385, (1993)

\bibitem{BKS} V. Baladi, A. Kondah, B. Schmitt, {\em Random correlations for small perturbations of expanding maps}, Random and Computational Dynamics, {\bf 4}, 179-204, (1996)

\bibitem{Bal97} V. Baladi, {\em Correlation spectrum of quenched and annealed equilibrium states for random expanding maps}, Commun. Math. Phys., {\bf 186}, 671-700, (1997)

\bibitem{Bal00} V. Baladi, {\em Positive transfer operators and decay of correlations}, Vol. 16, World Scientific, (2000)

\bibitem{Billingsley} P. Billingsley, {\em The Lindeberg-L\'evy theorem for martingales}, Proc. of the AMS, {\bf 12}, 788-792, (1961)

\bibitem{GB} A. Boyarsky, P. G\"{o}ra, {\em Absolutely continuous invariant measures for piecewise expanding $C^2$ transformations in $R^N$}, Israel J. Math., {\bf 67}, 272-286, (1989)

\bibitem{BG_book} A. Boyarsky, P. G\"{o}ra, {\em Laws of chaos. Invariant measures and dynamical systems in one dimension}, Probability and its applications, Birkhauser, (1997)

\bibitem{Brei} L. Breiman, {\em Probability}, Addison-Wesley, Reading, Mass, (1968)

\bibitem{Bro} A. Broise, {\em Etudes spectrales d'op\'erateurs de transfert et applications}, Ast\'erisque, {\bf 238}, (1996)

\bibitem{Buz99} J. Buzzi, {\em Exponential decay of correlations for random Lasota-Yorke maps}, Comm. Math. Phys., {\bf 208}, 25-54, (1999)

\bibitem{CC} J.-R. Chazottes, P. Collet {\em Almost sure central limit theorems and Erd\"{o}s-R\'enyi type law for expanding maps of the interval}, Ergodic Theory and Dynamical Systems, {\bf 25}, 419-441, (2005)

\bibitem{CCS} J.-R. Chazottes, P. Collet, B. Schmitt, {\em Statistical consequences of the Devroye inequality for processes. Applications to a class of non-uniformly hyperbolic dynamical systems}, Nonlinearity, {\bf 18}, 2341-2364, (2005)

\bibitem{CG} J.-R. Chazottes, S. Gou\"{e}zel, {\em Optimal concentration inequalities for dynamical systems}, Commun. Math. Phys., {\bf 316}, 843-889, (2012)

\bibitem{CK} N. Chernov, D. Kleinbock, {\em Dynamical Borel-Cantelli lemmas for Gibbs measures}, Israel J. Math., {\bf 122}, 1-27, (2001)

\bibitem{CMS} P. Collet, S. Martinez, B. Schmitt, {\em Exponential inequalities for dynamical measures of expanding maps of the interval}, Probab. Theory Relat. Fields, {\bf 123}, 301-322, (2002)

\bibitem{Conze_Raugi}  J.-P. Conze, A. Raugi, {\em Limit theorems for sequential expanding dynamical systems on [0,1]},  Ergodic theory and related fields, \textbf{89121}, Contemp. Math., 430, Amer. Math. Soc., Providence, RI, 2007.

\bibitem{CHMV} G. Cristadoro, N. Haydn, P. Marie, S. Vaienti, {\em Statistical  properties of intermittent maps with unbounded
derivative}, Nonlinearity, {\bf 23}, 1071-1096, (2010)

\bibitem{DZ} A. Dembo, O. Zeitouni, {\em Large Deviations, Techniques and Applications}, Applications of Mathematics {\bf 38}, Springer-Verlag, 2nd Edition, 1998

\bibitem{DN} M. Denker, M. Nicol, {\em Erd\"{o}s-R\'enyi laws for dynamical systems},  J. London Math. Soc., {\bf 87}, 497-508, (2012)

\bibitem{Dur} R. Durrett, {\em Probability : Theory and examples}, 4th edition, Cambridge Series in Statistical and Probabilistic Mathematics, 2010

\bibitem{El} R.S. Ellis, {\em Entropy, Large Deviations and Statistical Mechanics}, Springer-Verlag, New-York, 1985

\bibitem{ER} P. Erd\"{o}s, A. R\'enyi, {\em On a new law of large numbers}, J. Anal. Math., {\bf 23}, 103-111, (1970)

\bibitem{Go} M.I. Gordin, {\em The central limit theorem for stationary processes}, Soviet. Math. Dokl., {\bf 10}, 1174-1176, (1969)

\bibitem{G10} S. Gou\"{e}zel, {\em Almost sure invariance principle for dynamical systems by spectral methods}, Annals of Probability, {\bf 38}, (2010), 1639-1671

\bibitem{GH} Y. Guivarc'h, J. Hardy, {\em Th\'eor\`emes limites pour une classe de cha\^ines de Markov et applications aux diff\'eomorphismes d'Anosov}, Annales de l'I.H.P., section B, {\bf 24}, 73-98, (1988)
  

\bibitem{Hall_Heyde} P. Hall, C.C. Heyde, {\em Martingale limit theory and its application}, Probability and Mathematical Statistics, Academic Press, 1980, New York.

 \bibitem{HNVZ} N. Haydn, M. Nicol, S. Vaienti,  L. Zhang, {\em Central limit theorems for the shrinking target problem}, accepted for publication in J. Stat. Phys., (2013), DOI 10.1007/s10955-013-0860-3

\bibitem{H} H. Hennion, {\em Sur un th\'eor\`eme spectral et son application aux noyaux Lipschitziens}, Proceedings of the A.M.S., {\bf 118}, 627-634, (1993)

\bibitem{HH} H. Hennion, L. Herv\'e, {\em Limit theorems for Markov chains
and stochastic properties of dynamical systems by quasicompactness}, Lect. Notes in Math., {\bf 1766}, (2001), Springer-Verlag

\bibitem{HH2} H. Hennion, L.Herv\'e, {\em Central limit theorems for iterated random Lipschitz mappings}, Ann. Probab., {\bf 32}, 1934-1984, (2004)

\bibitem{HK} F. Hofbauer, G. Keller, {\em Ergodic properties of invariant measures for piecewise monotonic transformations}, Mathematische Zeitschrift, {\bf 180}, 119-140, (1982)

\bibitem{Hsieh} L.-Y. S. Hsieh, {\em Ergodic theory of multidimensional random dynamical systems}, Ph.D. Thesis, University of Victoria, (2008)

\bibitem{ITM} C.T. Ionescu-Tulcea, G. Marinescu, {\em Th\'eorie ergodique pour des classes d'op\'erations non compl\`etement continues}, Ann. Math., {\bf 52}, 140-147, (1950)

\bibitem{Ishi} H. Ishitani, {\em Central limit theorems for the random iterations of 1-dimensional transformations (Dynamics of complex systems)}, RIMS Kokyuroku, {\bf 1404}, 21-31, (2004)


\bibitem{Kel} G. Keller, {\em Generalized bounded variation and applications to piecewise monotonic transformations}, Z. Wahr. verw. Geb., {\bf 69}, 461-478, (1985)

\bibitem{Kessebohmer} J. Jaerisch, M. Kesseb\"ohmer and B. Stratmann, {\em A Fr\'echet law and an Erd\"os-Philipp law for maximal cuspidal windings}, Ergodic Theory and Dynamical Systems, {\bf 33:4} (2013) 1008-1028.
\bibitem{K86} Y. Kifer, {\em Ergodic theory for random transformations}, Boston : Birkhauser, (1986)

\bibitem{K88} Y. Kifer, {\em Random perturbations of dynamical systems}, Boston : Birkhauser, (1988)

\bibitem{K98} Y. Kifer, {\em Limit theorems for random transformations and processes in random environments},  Trans. Amer. Math. Soc., {\bf 350}, 1481-1518, (1998)

\bibitem{K08} Y. Kifer, {\em Thermodynamic formalism for random transformations revisited}, Stoch. Dyn., {\bf 08}, 77-102, (2008)

\bibitem{Kim} D. Kim, {\em The dynamical Borel-Cantelli lemma for interval maps}, Discrete Contin. Dyn. Syst., \textbf{17}, 891-900, (2007)

\bibitem{KY} E. Kobre, L.-S. Young, {\em Extended systems with deterministic local dynamics and random jumps}, Commun. Math. Phys., {\bf 275}, 709-720,, (2007)

\bibitem{LP} M.T. Lacey, W. Philipp, {\em A note on the almost sure central limit theorem}, Statistics and Probability Letters, {\bf 9}, 201-205, (1990)

\bibitem{LY} A. Lasota, J.-A. Yorke, {\em On the existence of invariant measures for piecewise monotonic transformations}, Trans. Amer. Math. Soc., {\bf 186}, 481-488, (1973)

\bibitem{Liv96} C. Liverani, {\em Central limit theorem for deterministic systems}, Pitman Research Notes in Mathematics Series, {\bf 362}, 56-75, (1996)

\bibitem{Liv13} C. Liverani, {\em Multidimensional expanding maps with singularities: a pedestrian approach}, Ergodic Theory and Dynamical Systems, {\bf 33}, 168-182, (2013)

\bibitem{Mal} C. Maldonado, {\em Fluctuation bounds for chaos plus noise in dynamical systems}, J. Stat. Phys., {\bf 148}, 548-564, (2012)

\bibitem{McDi1} C. McDiarmid, {\em On the method of bounded differences}, in {\em Surveys in Combinatorics}, London Math. Soc. Lecture Note Ser., {\bf 141}, Cambridge Univ. Press., 148-188, (1989)

\bibitem{McDi2} C. McDiarmid, {\em Concentration}, in {\em Probabilistic methods for algorithmic discrete mathematics}, Algorithms Combin., {\bf 16}. Springer, Berlin, 195-248, (1998)

\bibitem{MN} I. Melbourne, M. Nicol, {\em A vector-valued almost sure invariance principle for hyperbolic dynamical systems}, Ann. Probab., {\bf 37}, 478-505, (2009)

\bibitem{Mor1} T. Morita, {\em Random iteration of one-dimensional transformations}, Osaka J. Math., {\bf 22}, 489-518, (1985)

\bibitem{Mor2} T. Morita, {\em Deterministic version lemmas in ergodic theory of random dynamical systems}, Hiroshima Math. J., {\bf 18}, 15-29, (1988)

\bibitem{Mor3} T. Morita, {\em A generalized local limit theorem for Lasota-Yorke transformations}, Osaka J. Math., {\bf 26}, 579-595, (1989)

\bibitem{Pel} S. Pelikan, {\em Invariant densities for random maps of the interval}, Trans. Amer. Math. Soc., {\bf 281}, 813-825, (1984)

\bibitem{R-E} J. Rousseau-Egele, {\em Un th\'eor\`eme de la limite locale pour une classe de transformations dilatantes et monotones par morceaux}, Annals of Probability, {\bf 11}, 772-788, (1983)

\bibitem{S} B. Saussol, {\em Absolutely continuous invariant measures for multidimensional expanding maps},  Israel J. Math., {\bf 116}, 223-248, (2000)

\bibitem{Sprindzuk} V.G. Sprindzuk, {\em Metric theory of Diophantine approximations}, V. H. Winston and Sons, Washington,  D.C., 1979, Translated from the Russian and edited by Richard
  A. Silverman, With a foreword by Donald J. Newman, Scripta Series in Mathematics. MR MR548467 (80k:10048).

\bibitem{Tho} D. Thomine, {\em A spectral gap for transfer operators of piecewise expanding maps}, Discrete Contin. Dyn. Syst., \textbf{30}, 917-944, (2011)

\bibitem{Tsu} M. Tsujii, {\em Absolutely continuous invariant measures for expanding piecewise linear maps}, Inventiones Mathematicae, {\bf 143}, 349-373, (2001)

\bibitem{Tum} F. T\"{u}mel, {\em Random walks on a lattice with deterministic local dynamics}, Ph.D thesis, University of Houston, (2012)

\end{thebibliography}
\end{document}